\documentclass[10pt]{amsart}

\usepackage[latin1]{inputenc}
\usepackage[T1]{fontenc}
\usepackage{textcomp}
\usepackage{graphicx}
\usepackage{epsfig}
\usepackage{amsmath}
\usepackage{amssymb}
\usepackage{amsthm}
\usepackage{booktabs}
\usepackage{stmaryrd}
\usepackage{url}
\usepackage{longtable}
\usepackage[figuresright]{rotating}
\usepackage[all]{xy}
\usepackage{pdfsync}
\usepackage[toc,page]{appendix}
\usepackage{fullpage}
\usepackage{color}
\usepackage{hyperref}

\newcommand{\complex}{\mathbb{C}}

\newcommand{\ganz}{\mathbb{Z}}

\newcommand{\q}{\mathbb Q}

\newcommand{\diagram}[1]{\begin{displaymath}
\xymatrix{#1}
\end{displaymath}}
\newcommand{\diagramnr}[2]{\begin{equation}
\label{#1}
\xymatrix{#2} 
\end{equation}}

\newcommand{\real}{\mathbb R}

\newcommand{\eq}[1]{\begin{align}#1\end{align}}
\newcommand{\eqst}[1]{\begin{align*}#1\end{align*}}
\newcommand{\okhat}{\widehat{\mathcal O }_K}
\newcommand{\olhat}{\widehat{\mathcal O }_L}
\newcommand{\kab}{Gal(K^{ab}/K)}
\newcommand{\iok}{ I_K}
\newcommand{\ok}{\mathcal O _K}
\newcommand{\iol}{ I_L}
\newcommand{\ol}{\mathcal O _L}
\newcommand{\lab}{Gal(L^{ab}/L)}

\theoremstyle{plain}
\newtheorem{theorem}{Theorem}[section]

\newtheorem{remark}[theorem]{Remark}
\newtheorem{definition}[theorem]{Definition}
\newtheorem{proposition}[theorem]{Proposition}

\newtheorem{lemma}[theorem]{Lemma}
\newtheorem{corollary}[theorem]{Corollary}

\newcommand\bp{\begin{proof}}
\newcommand\ep{\end{proof}}

\newcommand\aaa{\mathfrak a}
\newcommand\pp{\mathfrak p}

\newcommand{\N}{\mathbb N}
\newcommand{\C}{\mathbb C}

\newcommand\ohs{{\widehat{\OO}^\times}}
\newcommand\OO{{\mathcal O}}

\numberwithin{equation}{section}

\title{On arithmetic models and functoriality of Bost-Connes systems. With an appendix by Sergey Neshveyev}

\date{\today}

\begin{document}

\author[B. Yalkinoglu]{Bora Yalkinoglu}
\address{Max-Planck Institute for Mathematics, Bonn}
\email{boray@mpim-bonn.mpg.de}
\thanks{This work has been supported by the Marie Curie Research Training Network MRTN-CT-2006-031962 in Noncommutative Geometry, EU-NCG}

\begin{abstract} 
This paper has two parts. In the first part we construct arithmetic models of Bost-Connes systems for arbitrary number fields, which has been an open problem since the seminal work of Bost and Connes \cite{bc}. In particular our construction shows how the class field theory of an arbitrary number field can be realized through the dynamics of a certain operator algebra. This is achieved by working in the framework of Endomotives, introduced by Connes, Marcolli and Consani \cite{CCMend}, and using a classification result of Borger and de Smit \cite{BdS} for certain $\Lambda$-rings in terms of the Deligne-Ribet monoid. Moreover the uniqueness of the arithmetic model is shown by Sergey Neshveyev in an appendix. In the second part of the paper we introduce a base-change functor for a class of algebraic endomotives and construct in this way an algebraic refinement of a functor from the category of number fields to the category of Bost-Connes systems, constructed recently by Laca, Neshveyev and Trifkovic \cite{LNT}.
\end{abstract}

\maketitle

\setcounter{tocdepth}{1}

\section*{Introduction}
\noindent
In this paper we show the existence of arithmetic models of Bost-Connes systems for arbitrary number fields, which was an open problem going back to the work of Bost and Connes \cite{bc}. We also introduce an algebraic refinement of a functor from the category of number fields to the category of Bost-Connes systems constructed recently by Laca, Neshveyev and Trifkovic \cite{LNT}. \\ \\
For every number field $K$ there is a $C^*$-dynamical system (cf., section \ref{bctypesystems}) \eqst{\mathcal A _K = (A_K,\sigma_t)} called the \textbf{Bost-Connes system} or \textbf{BC-system} of $K$. Such a system was first constructed in the case of the rational field in the seminal paper \cite{bc}, and later for arbitrary number fields by Ha and Paugam \cite{HaPa}.
Among the most interesting properties of BC-systems are the following four properties. \bigskip
\begin{itemize}
\item[(i)] The partition function of $\mathcal A$ is given by the Dedekind zeta function of $K$. \smallskip
\item[(ii)] The maximal abelian Galois group $Gal(K^{ab}/K)$ of $K$ acts as symmetries on $\mathcal A$. \smallskip
\item[(iii)] For each inverse temperature $0 < \beta \leq 1$ there is a unique $KMS_\beta$-state. \smallskip
\item[(iv)] For each $\beta > 1$ the action of the symmetry group $Gal(K^{ab}/K)$ on the set of extremal $KMS_\beta$-states is free and transitive. \bigskip
\end{itemize}
In \cite{HaPa} property (i) and (ii) were shown to hold for all BC-systems $\mathcal A _K$. The difficult problem of classifying the $KMS_\beta$-states of BC-systems was solved by Laca, Larsen and Neshveyev \cite{LLN} by building upon earlier work of \cite{bc}, Connes and Marcolli \cite{ConMarGL2}, Laca \cite{Laca1} and Neshveyev \cite{Serg1}, thus proving property (iii) and (iv) for all $\mathcal A _K$. From an arithmetic view point the most interesting property of BC-systems is the existence of arithmetic models. A BC-system $\mathcal A _K$ has an \textbf{arithmetic model} if there exists a $K$-rational subalgebra $A_K^{arith}$ of $A_K$, called an \textbf{arithmetic subalgebra}, such that the following additional three properties are satisfied (see, e.g., \cite{cmr}) \bigskip
\begin{itemize}
\item[(v)] For every extremal $KMS_\infty$-state $\varrho$ and every $f \in A_K^{arith}$, we have \eqst{\varrho(f) \in K^{ab}} and further $K^{ab}$ is generated over $K$ by these values.
\item[(vi)] If we denote by ${}^{\nu}\hspace{-0.6mm}\varrho$ the action of a symmetry $\nu \in \kab$ on an extremal $KMS_\infty$-state $\varrho$ (given by pull-back), we have for every element $f \in A_K^{arith}$ the following compatibility relation \eqst{{}^{\nu}\hspace{-0.6mm}\varrho(f)  = \nu^{-1}(\varrho(f))\,.} 
\item[(vii)] The $\complex$-algebra $A_K^{arith}$$\otimes_K$$\complex$ is dense in $A_K$. \bigskip
\end{itemize}
The existence of an algebraic model of $\mathcal A _\q$ was already shown in \cite{bc}. Ten years later
Connes, Marcolli and Ramachandran \cite{cmr} constructed in a beautiful way arithmetic models of $\mathcal A _K$ in the case of $K$ imaginary quadratic by drawing the connection to the theory of Complex Multiplication on the modular curve by using the $GL_2$-system of \cite{ConMarGL2}. A first approach towards the construction of (partial) arithmetic models of more general BC-systems $\mathcal A_K$ was undertaken in \cite{Y1} where the theory of Complex Multiplication on general Siegel modular varieties and the $GSp_{2n}$-systems of \cite{HaPa} were used to construct partial arithmetic models of $\mathcal A _K$ in the case of $K$ containing a CM field. This approach exhausted at the same time the full power of the existing explicit class field theory (which is only known for $K=\q$ or imaginary quadratic, and partially for $K$ a CM field). \\ \\
The natural question whether all BC-systems $\mathcal A _K$ posses an algebraic model proved to be resistant at first. In the case of the classical BC-system it was shown by Marcolli \cite{matcycl} that $\mathcal A _\q$ can be described in the context of endomotives, introduced by Connes, Consani and Marcolli \cite{CCMend}, and the theory of $\Lambda$-rings, i.e., rings with a commuting family of Frobenius lifts as extra structure. We will show in our work that this approach is in fact the correct one for the general case. An elegant classification result of Borger and de Smit \cite{BdS} of certain $\Lambda$-rings in terms of the Deligne-Ribet monoid paves the way for the case of arbitrary number fields. \\ \\
More precisely, for every number field $K$ the results of \cite{BdS} allow us to construct an algebraic endomotive (cf., \ref{defalgek}) \eqst{\mathcal E _K = E_K \rtimes I_K} over $K$, where the $K$-algebra $E_K$ is a direct limit $\varinjlim E_\mathfrak f$ of finite, \'etale $K$-algebras $E_\mathfrak f$ which come from a refined Grothendieck-Galois correspondence in terms of the Deligne-Ribet monoid $\mathrm{DR}_K$ (see Corollary \ref{ilggc}). The monoid of (non-zero) integral ideals $\iok$ of $K$ is acting by Frobenius lifts on $E_K$. \\
In general there is a functorial way of attaching to an algebraic endomotive $\mathcal E$ a $C^*$-algebra $\mathcal E ^{an}$, containing $\mathcal E$, which is called the analytic endomotive of $\mathcal E$. Moreover, in good situations $\mathcal E$ determines naturally a time evolution $\sigma : \real \to Aut(\mathcal E^{an})$ on $\mathcal E ^{an}$ by means of Tomita-Takesaki theory, so that we end up with a $C^*$-dynamical system \eqst{\mathcal E ^{mean} = (\mathcal E ^{an},\sigma_t)\, ,} depending only on $\mathcal E$ called the measured analytic endomotive of $\mathcal E$ (cf., section \ref{endomotives}). Our first main result will be \smallskip
\begin{theorem}
\label{ThmA}
For every number field $K$ the measured analytic endomotive $\mathcal E _K^{mean}$ of the algebraic endomotive $\mathcal E _K$ exists and is in fact naturally isomorphic to the BC-system $\mathcal A _K$. \smallskip 
\end{theorem} \noindent
The key observations for proving this theorem are proposition \ref{keyprop}, which shows that the Deligne-Ribet monoid $\mathrm{DR}_K$ is naturally isomorphic to $\okhat \times_{\okhat ^\times} \kab$, and proposition \ref{equalmeasures}, which shows that the time evolutions of both systems agree. \\ \\
The most important result of our paper is to show that all $\mathcal A _K$ posses an arithmetic model. \smallskip
\begin{theorem}
\label{ThmB} For all number fields $K$ the BC-systems $\mathcal A _K$ (resp. $\mathcal E _K^{mean}$) posses an arithmetic model with arithmetic subalgebra given by the algebraic endomotive $\mathcal E _K = E_K \rtimes I_K$. \smallskip
\end{theorem} \noindent
The proof of this theorem relies on the fact that the algebras $E_\mathfrak f$ defining the algebraic endomotive $\mathcal E _K$ are finite products of strict ray class fields of $K$ (cf., \eqref{arithalg}). In particular our main result shows that the class field theory of an arbitrary number field can be realized through the dynamics of an operator algebra. \\ \\
In the appendix (section \ref{sergey}) Sergey Neshveyev has shown moreover that under very natural conditions, satisfied by our arithmetic subalgebra, the arithmetic model of a BC-system is in fact unique, see Theorem \ref{SergeyA} and \ref{SergeyB}.
\\ \\
The second part of our paper will be concerned with functoriality properties of Bost-Connes systems.
Recently, Laca, Neshveyev and Trifkovic \cite{LNT} constructed a functor from the category of number fields to an appropriate category\footnote{Morphisms in this category are given by Hilbert $C^*$-bimodules, or in other words by $C^*$-correspondences} of BC-systems. We will show that their functor does fit naturally into the framework of endomotives by constructing an algebraic refinement of their functor. For this we develop the notion of base-change for our algebraic endomotives $\mathcal E _K$ which is rooted in the functoriality properties of Artin's reciprocity map and certain properties of the Deligne-Ribet monoid (see section \ref{basechange}). Using this tool we define a functor from the category of number fields to the category of algebraic endomotives over $\q$ by sending a field $K$ to the base-changed endomotive $\mathcal E _K ^\q$ and a morphism of number fields $K\to L$ to an algebraic bimodule $\mathcal Z _K^L$ (cf., \eqref{algmorph}). Our third main result reads then as follows  \smallskip
\begin{theorem}
\label{ThmC}
The functor defined by $K \mapsto \mathcal E _K ^\q$ and $(K\to L) \mapsto \mathcal Z _K^L$, as above recovers, by passing to the analytic endomotive, the functor constructed by Laca, Neshveyev and Trifkovic \cite{LNT}. \smallskip 
\end{theorem}

\subsection*{Acknowledgements}
This paper was written while the author was a PhD student at the Universit\'e Pierre et Marie Curie in Paris. I would like to thank my advisor Eric Leichtnam for his support and many useful comments  on this paper. Further, the author would like to thank James Borger for sharing the preprint \cite{BdS} and interesting discussions about $\Lambda$-rings, Eugene Ha for many useful suggestions, Matilde Marcolli for her interest and encouraging the author to work on the ideas presented in this paper and Sergey Neshveyev for his interest and for providing a very nice appendix about the uniqueness of arithmetic models. Finally, I would like to thank the referee whose valuable comments have helped to improve the exposition of this paper.

\tableofcontents

\part{Arithmetic subalgebras}
\noindent
Before we explain and perform our construction of arithmetic subalgebras in form of the algebraic endomotives $\mathcal E _K$, we will briefly recall the definition and properties of the systems $\mathcal A _K$ of Ha and Paugam, present the theory of endomotives to an extent sufficient for our applications and explain then in some detail the Deligne-Ribet monoid $\mathrm{DR}_K$ which will be an object of central importance for the construction of the algebraic endomotives $\mathcal E _K$ and the classification result of Borger and de Smit.

\subsection*{Notations and Conventions}

$K$ will always denote a number field with ring of integers $\ok$. Further, we fix an embedding $K \subset \complex$ and consider the algebraic closure $\overline K$ of $K$ in $\complex$. The maximal abelian algebraic extension of $K$ is denoted by $K^{ab}$. By $\iok$ we denote the monoid of (non-zero) integral ideals of $\ok$ and by $J_K$ the group of fractional ideals of $K$. As usual, we write $\mathbb A _K = \mathbb A _{K,f} \times \mathbb A _{K,\infty}$ for the adele ring of $K$ with $\mathbb A _{K,f}$ the finite adeles and $\mathbb A _{K,\infty}$ the infinite adeles. If $R$ is a ring, we denote by $R^\times$ its group of invertible elements. Invertible adeles are called ideles. By $\okhat \subset \mathbb A _K$ we denote the finite, integral adeles of $K$, further we set $\okhat^\natural = \mathbb A _{K,f}^\times \cap \okhat$. We denote Artin's reciprocity map by $[\cdot]_K : \mathbb A _K ^\times \to \kab$. Usually we omit the subscript ${}_K$ and write only $[\cdot]$. Moreover we denote the idele norm by $N_{K/\q} :\mathbb A _{K,f}^\times \to \mathbb A _{\q,f}^\times$ which induces in particular the norm maps $N_{K/\q} : J_K \cong \mathbb A _{K,f}^\times / \okhat ^\times \to \q$ and $N_{K/\q} : I_K \cong \okhat^\natural / \okhat^\times \to \ganz$. Also, we use the delta function $\delta_{a,b} = \left\{ \begin{array}{cc} 1 & \text{if } a=b \\ 0 & \text{otherwise}  \end{array} \right.$. Finally, we denote by $|X|$ the cardinality of a set $X$ and, for another set $Y$, we write $X\sqcup Y$ for their disjoint union.

\section{BC-systems}
\label{bctypesystems}
\noindent
Let us recall the definition of the $C^*$-dynamical systems $\mathcal A _K$ and some of its properties, following \cite{LLN}. Consider the topological space \eq{Y_K = \okhat \times_{\okhat^\times}\kab} defined as the quotient space of the direct product $\okhat \times \kab$ under the action of $\okhat^\times$ given by \eqst{s \cdot (\rho,\alpha) = (\rho s , [s]^{-1}\alpha)\,.}
There are two natural actions on $Y_K$. On the one hand, the monoid $\iok \cong \okhat^\natural / \okhat^\times$ of (non-zero) integral ideals of $K$ acts by \eq{\label{obviousaction}s\cdot[\rho,\alpha] = [\rho s , 
[s]^{-1}\alpha]\, ,} and, on the other hand, the maximal abelian Galois group $\kab$ acts by \eqst{\gamma \cdot [\rho,\alpha] = [\rho, \gamma \alpha]\,.}
The first action gives rise to the semigroup crossed product $C^*$-algebra \eq{\label{bcalgebra}A_K = C(Y_K)\rtimes \iok\, ,}
and together with the time evolution defined by \eq{\sigma_t(fu_s) = N_{K/\q}(s)^{it}fu_s\, ,}
where $f \in C(Y_K)$ and $u_s$ the isometry encoding the action of $s \in \iok$, we end up with the BC-system of $K$ in form of the $C^*$-dynamical system \eq{\label{bctype}\mathcal A_K = (A_K,\sigma_t)\,.}
Moreover, the action of the Galois group $\kab$ on $Y_K$ induces naturally a map \eqst{\kab \longrightarrow \mathrm{Aut}(\mathcal A _K)\,.}
Later we will need the classification of extremal $\sigma$-$KMS_\beta$-states, as given elegantly in \cite{LLN}, at $\beta = 1$ and $\beta = \infty$. The approach of \cite{LLN} relates $KMS_\beta$-states of $\mathcal A _K$ to measures on $Y_K$ with certain properties. We recommend the reader to consult their paper.
\subsection{Classification at $\beta =1$}
\label{measure1}
\noindent
In the proof of \cite{LLN} Theorem 2.1 it is shown that the unique $KMS_1$-state of $\mathcal A _K$ corresponds to the measure $\mu_1$ on $Y_K$ which is given by the push-forward (under the natural projection) of the product measure \eqst{\prod_{\mathfrak p}\mu_\mathfrak p \times \mu_\mathcal G} on $\okhat \times \kab$,
where $\mu_\mathcal G$ is the normalized Haar measure on $\kab$, and $\mu_\mathfrak p$ is the additive normalized Haar measure on $\mathcal O _{K_\mathfrak p}$. Equivalently, it is shown that $\mu_1$ is the unique measure on $Y_K$ satisfying $\mu_1(Y_K) = 1$, and the scaling condition \eq{\label{scalingcond} \mu_1(gZ) = \mathcal N_{K/\q}(g)^{-1}\mu_1(Z)\, ,} for every Borel subset $Z \subset Y_K$ and $g \in J_K = \mathbb A _{K,f}^\times / \okhat^\times$ such that $gZ \subset Y_K$\footnote{The action of the group $J_K \cong \mathbb A _{K,f}^\times / \okhat^\times$ of fractional ideals of $K$ on $Y_K$ is the one given in \eqref{obviousaction}.}.

\subsection{Classification at $\beta = \infty$}
\label{kmsinfty}
\noindent
The set of extremal $KMS_\infty$-states of $\mathcal A _K$ is parametrized by the subset $Y_K^\times = \okhat^\times \times_{\okhat^\times} \kab$ of $Y_K$, and for $\omega \in Y_K^\times$ the corresponding extremal $KMS_\infty$-state $\varphi_\omega$ is given by \eqst{\varphi_\omega(fu_s) = \delta_{s,1}f(\omega)\,.}
In other words, extremal $KMS_\infty$-states of $\mathcal A _K$ correspond to Dirac measures on $Y_K$ with support in $Y_K^\times$.

\section{Endomotives}
\label{endomotives} \noindent
We will recall briefly the theory of endomotives, following our main reference \cite{ConMar}. Endomotives come in three different flavours: algebraic, analytic and measured analytic. Each aspect could be developed independently, but for our purposes, it is enough to concentrate on algebraic endomotives, and show how to associate an analytic and a measured analytic endomotive to it. \\ \\
Recall that we fixed an embedding $ K \subset \complex$ and understand $\overline K$ to be the algebraic closure of $K$ in $\complex$.

\subsection{Algebraic flavour}
\noindent
We denote by $\mathfrak E_K$ the category of finite dimensional, \'etale $K$-algebras with morphisms given by $K$-algebra homomorphisms.
Let $((A_i)_{i \in I},S)$ be a pair consisting of an inductive system $(A_i)_{i\in I}$ (with transition maps $\xi_{i,j}$ for $i\leq j$) in $\mathfrak E_K$ and an abelian semigroup $S$ acting on the inductive limit $A = \varinjlim_i A_i$ by $K$-algebra endomorphisms. We don't require the action of $S$ to respect the levels $A_i$ or to be unital, so in general $e = \rho(1)$, for $\rho \in S$, will only be an idempotent, i.e., $e^2 = e$. Moreover, we assume that every $\rho \in S$ induces an isomorphism of $K$-algebras $\rho : A \overset\cong\longrightarrow eAe = eA$. 
\begin{definition}
Let $((A_i),S)$ be a pair like above. Then the associated algebraic endomotive $\mathcal E$ is defined to be the associative, unital $K$-algebra given by the crossed product \eqst{\mathcal E = A\rtimes S \,.}
\end{definition} \noindent
The algebraic endomotive $\mathcal E$ can be described explicitly in terms of generators and relations by adjoining to $A$ new generators $U_\rho$ and $U^*_\rho$, for $\rho \in S$, and imposing the relations 

\begin{equation*}\begin{array}{ccc}U_\rho^*U_\rho = 1, & U_\rho U_\rho^* = \rho(1), &\forall \ \rho \in S \\ U_{\rho_1}U_{\rho_2}=U_{\rho_1\rho_2}, &U^*_{\rho_2\rho_1}=U^*_{\rho_1}U^*_{\rho_2}, & \forall \ \rho_1,\rho_2 \in S \\ U_\rho a = \rho(a) U_\rho, & aU^*_\rho = U^*_\rho \rho(a), & \forall \ \rho\in S, \forall \ a \in A \, .
\end{array} \end{equation*}

\begin{lemma}[Lemma 4.18 \cite{ConMar}]
\label{struclem}
1) The algebra $\mathcal E$ is the linear span of the monomials $U^*_{\rho_1}aU_{\rho_2}$, for $a \in A$ and $\rho_1,\rho_2 \in S$. \\
2) The product $U_g = U^*_{\rho_2}U_{\rho_1}$ only depends on the ratio $g=\rho_1 / \rho_2$ in the group completion $\widetilde S$ of $S$. \\
3) The algebra $\mathcal E$ is the linear span of the monomials $aU_g$, for $a \in A$ and $g \in \widetilde S$.
\end{lemma}

\begin{remark}
Equivalently one can rephrase the theory of algebraic endomotives in the language of Artin motives. Namely, every finite, \'etale $K$-algebra $B$ gives rise to a zero-dimensional variety $\mathrm{Spec}(B)$, or in other words, to an Artin motive. This coined the term "endomotive".
\end{remark}

\subsection{Analytic flavour}
\label{anflavour}
\noindent
Given an algebraic endomotive $((A_i),S)$, we obtain a topological space $\mathcal X$ defined by the projective limit \eqst{\mathcal X = \varprojlim_i \mathrm{Hom}_{K\text{-alg}}(A_i,\overline K)\, ,}
which is equipped with the profinite topology, i.e., $\mathcal X$ is a totally disconnected compact Hausdorff space\footnote{In other words $\mathcal X$ is given by the $\overline K$-points of the provariety $\varprojlim Spec(A_i)$}. Using $\mathcal X \cong \mathrm{Hom}_{K\text{-alg}}(\varinjlim A_i,\overline K) = \mathrm{Hom}_{K\text{-alg}}(A,\overline K)$, we see in particular that each $\rho \in S$ induces a homeomorphism $\rho : \mathcal X^e = \mathrm{Hom}(eA,\overline K)\longrightarrow \mathcal X$ by $\chi \in \mathcal X^e \mapsto \chi \circ \rho \in \mathcal X$, where $e =\rho(1)$. In this way, we get an action of $S$ on the abelian $C^*$-algebra $C(\mathcal X)$ by endomorphisms \begin{equation*}\phi(f)(\chi)= \left\{ \begin{array}{cc} 0 & \text{if }\chi(e)=0 \\ f(\chi\circ \rho) &\text{if }\chi(e)=1  \end{array}\right. \end{equation*} 
and we can consider the semigroup crossed product $C^*$-algebra (see, e.g., \cite{Laca2}) \eqst{\mathcal E^{an} = C(\mathcal X)\rtimes S\, ,}
which we define to be the \textbf{analytic endomotive} of the algebraic endomotive $((A_i),S)$. Using the embedding $\iota : K \to \complex$ we obtain an embedding of commutative algebras $A \hookrightarrow C(\mathcal X)$ by \eqst{a \mapsto \mathrm{ev}_a : \chi \mapsto \chi(a)\, ,}
and this induces an embedding of algebras \eqst{\mathcal E = A \rtimes S \hookrightarrow C(\mathcal X)\rtimes S\,.}
The algebraic endomotive is said to give an \textit{arithmetic structure} to the analytic endomotive $\mathcal E ^{an}$.

\subsubsection{Galois action}
\noindent
The natural action of the absolute Galois group $\mathrm{Gal}(\overline K /K)$ on $\mathcal X = \mathrm{Hom}(A,\overline K)$ induces an action of $\mathrm{Gal}(\overline K /K)$ on the analytic endomotive $\mathcal E^{an}$ by automorphisms preserving the abelian $C^*$-algebra $C(\mathcal X)$ and fixing the $U_\rho$ and $U_\rho^*$. Moreover, the action is compatible with pure states on $\mathcal E^{an}$ which do come from $C(\mathcal X)$ in the following sense (see Prop.~4.29 \cite{ConMar}). For every $a \in A$, $\alpha \in \mathrm{Gal}(\overline K / K)$, and any pure state $\varphi$ on $C(\mathcal X)$, we have $\varphi(a) \in \overline K$ and \eqst{\alpha(\varphi(a)) = \varphi(\alpha^{-1}(a))\,.}
Moreover, it is not difficult to show (see Prop.~4.30 \cite{ConMar}) that in case where all the $A_i$ are finite products of abelian, normal field extensions of $K$, as in our applications later on,  
the action of $\mathrm{Gal}(\overline K / K)$ on $\mathcal E^{an}$ descends to an action of the maximal abelian quotient $\mathrm{Gal}(K^{ab}/K)$.

\subsection{Measured analytic flavour}
\noindent
Let us start again with an algebraic endomotive $((A_i),S)$. On every finite space $\mathcal X_i = \mathrm{Hom}(A_i,\overline K)$, we can consider the normalized counting measure $\mu_i$.
We call our algebraic endomotive \textit{uniform}, if the transition maps $(\xi_{i,j})$ are compatible with the normalized counting measures, i.e., \eqst{\mu_i = (\xi_{i,j})_*\mu_j \text{, for all } i\leq j\,.} In this case the $\mu_i$ give rise to a projective system of measures and induce a propability measure $\mu$, the so-called Prokhorov extension, on $\mathcal X = \varprojlim \mathcal X _i$ (compare p. 545 \cite{ConMar}). 

\subsubsection{A time evolution}
\label{ttcon}
\noindent
Now, let us write $\varphi = \varphi_\mu$ for the corresponding state on the analytic endomotive $\mathcal E ^{an} = C(\mathcal X)\rtimes S$  given by \eqst{\varphi(fu_s) = \delta_{s,1} \int_\mathcal X f d\mu \,.}
The GNS construction gives us a representation $\pi_\varphi$ of $\mathcal E^{an}$ on a Hilbert space $\mathcal H_\varphi$ (depending only on $\varphi$). Further, we obtain a von Neumann algebra $\mathcal M _\varphi$ as the bicommutant of the image of $\pi_\varphi$, and, under certain technical assumptions on $\varphi$ (see pp. 616 \cite{ConMar}), the theory of Tomita-Takesaki equips $\mathcal M_\varphi$ with a time evolution $\sigma^{\varphi} : \real \to Aut(\mathcal M _\varphi)$, the so-called modular automorphism group. Now, if we \textbf{assume} that $\pi_\varphi$ is faithful, and moreover, the time evolution $\sigma^\varphi$ respects the $C^*$-algebra $C(\mathcal X) \rtimes S \cong \pi_\varphi(C(\mathcal X)\rtimes S) \subset \mathcal M_\varphi$, we end up with a $C^*$-dynamical system \eqst{\mathcal E ^{mean} = (C(\mathcal X)\rtimes S,\sigma^{\varphi})\, ,}
which we call a measured analytic endomotive. If it exists, it only depends on the (uniform) algebraic endomotive we started with.

\section{The Deligne-Ribet monoid}
\noindent
We follow \cite{DelRib} and \cite{BdS} in this section. Recall that $\iok$ denotes the monoid of (non-zero) integral ideals of our number field $K$. For every $\mathfrak f \in \iok$ we define an equivalence relation $\sim_\mathfrak f$ on $\iok$ by \eqst{\mathfrak a \sim_\mathfrak f \mathfrak b :\Leftrightarrow \exists \ x \in K^\times_+ \cap (1+\mathfrak f \mathfrak b^{-1}) : (x) = \mathfrak a \mathfrak b^{-1}\, ,}
where $K^\times_+$ denotes the subgroup of totally positive units in $K$ and $(x)$ the fractional ideal generated by $x$. The quotient \eqst{DR_\mathfrak f = \iok / \sim_\mathfrak f} is a finite monoid under the usual multiplication of ideals. We denote elements in $DR_\mathfrak f$ by $[\mathfrak a]_\mathfrak f$, with $\mathfrak a \in I_K$, but often we omit the bracket $[\cdot]_\mathfrak f$. Moreover, for every $\mathfrak f \mid \mathfrak f'$, we obtain a natural projection map \eq{\label{projmaps}\pi_{\mathfrak f,\mathfrak f'}:DR_{\mathfrak f'} \longrightarrow DR_\mathfrak f \text{, \  \ \ \ \ \ } [\mathfrak a]_\mathfrak {f'} \longmapsto [\mathfrak a]_\mathfrak f} and thus a projective system $(I_\mathfrak f)_{\mathfrak f \in \iok}$, whose limit \eqst{\mathrm{DR}_K = \varprojlim_\mathfrak f DR_\mathfrak f} is a (topological) monoid\footnote{We take the profinite topology.}, called the \textbf{Deligne-Ribet monoid} of $K$. 

\subsection{Some properties of $\mathrm{DR}_K$} 
\noindent
First, we have to recall some notations. A cycle $\mathfrak h$ is given by a product $\prod_\mathfrak p \mathfrak p ^{n_\mathfrak p}$ running over all primes of $K$, where the $n_\mathfrak p$'s are non-negative integers, with only finitely many of them non-zero. Further $n_{\mathfrak p} \in \{0,1\}$ for real primes, and $n_\mathfrak p = 0$ for complex primes. The finite part $\prod_{\mathfrak p \nmid \infty} \mathfrak p ^{n_\mathfrak p}$ can be viewed as an element in $\iok$. Moreover, we write $(\infty)$ for the cycle $\prod_{\mathfrak p \text{ real}} \mathfrak p$. \\ \\
If we denote by $C_\mathfrak f$ the (strict) ray class group of $K$ associated with the cycle $\mathfrak f(\infty)$, for $\mathfrak f \in \iok$, one can show that \eq{\label{prop1} DR_\mathfrak f ^\times = C_\mathfrak f\, ,} i.e., the group of invertible elements $DR_\mathfrak f ^\times$ can be identified naturally with $C_{\mathfrak f}$ (see (2.6) \cite{DelRib}). By class field theory, we know that the strict ray class group $C_\mathfrak f$ can be identified with the Galois group (over $K$) of the strict ray class field $K_\mathfrak f$ of $K$, i.e., we have \eqst{C_\mathfrak f \cong \mathrm{Gal}(K_\mathfrak f / K)\,.} 
As an immediate corollary, we obtain \eq{\label{prop3} \mathrm{DR}_K^\times = \varprojlim_{\mathfrak f} C_\mathfrak f \cong \varprojlim_\mathfrak f \mathrm{Gal}(K_\mathfrak f / K) \cong \mathrm{Gal}(K^{ab}/K)\, ,}
i.e., using class field theory, we can identify the invertible elements of $\mathrm{DR}_K$ with the maximal abelian Galois group of $K$.  
Moreover, we have the following description \eq{\label{prop2} DR_\mathfrak f \cong \coprod_{ \mathfrak d \mid \mathfrak f} C_{\mathfrak f / \mathfrak d}\, ,} where an element $\mathfrak a \in C_{\mathfrak f / \mathfrak d}$ is sent to $\mathfrak a \mathfrak d \in DR_\mathfrak f$ (see the bottom of p. 239 \cite{DelRib} or \cite{BdS}). \\
There is an important map of topological monoids \eq{\label{propdef1} \iota : \okhat \longrightarrow \mathrm{DR}_K}
given as follows: For $m_\mathfrak f \in \mathcal O _K / \mathfrak f$, we choose a lift $m_\mathfrak f^+ \in \mathcal O _{K,+}$, and map this to the ideal $(m_\mathfrak f^+) \in DR_\mathfrak f$. The map $\iota$ is then defined by \eqst{(m_\mathfrak f) \in \varprojlim_{\mathfrak f} \mathcal O _K / \mathfrak f \cong \okhat \longmapsto ( (m_{\mathfrak f}^+) ) \in \varprojlim_\mathfrak f DR_\mathfrak f = \mathrm{DR}_K\, ,}
which can be shown to be independent of the choice of the liftings (Prop. 2.13 \cite{DelRib}). \\
Let us denote by $U_K^+$ the closure of the totally positive units $\mathcal O _{K,+}^\times = \ok^\times \cap K^\times_+$ in $\okhat^\times$.
\begin{proposition}[Prop. 2.15 \cite{DelRib}]
\label{iotakernel}
Let $\rho,\rho' \in \okhat$. Then $\iota(\rho) = \iota(\rho')$ if and only if $\rho = u\rho' $ for some $u \in U_K^+$.
\end{proposition} \noindent
Therefore, it makes sense to speak of $\iota$ having kernel $U_K^+$. Moreover, if we denote by $(\rho) \in \iok$ (resp.~$[\rho] \in \kab$) the ideal generated by an idele (resp.~the image under Artin reciprocity's map), then we have the following:
\begin{proposition}[Prop. 2.20 and 2.23 \cite{DelRib}]
\label{centralprop}
For $\rho \in \okhat^\natural$, we have \eqst{\iota(\rho) = (\rho) [\rho]^{-1} \in \mathrm{DR}_K\,.}
In particular, for $\rho \in \okhat^\times$, we obtain \eqst{\iota(\rho) = [\rho]^{-1} \in \mathrm{DR}_K^{\times}\,.} 
\begin{remark}
\label{intersection}
The reader should keep in mind, that the intersection $\iok \cap \mathrm{DR}_K^\times$ is trivial.
\end{remark}
\end{proposition}

\section{A classification result of Borger and de Smit}
\label{bds} \noindent
The results in this section are based on the preprint \cite{BdS} of Borger and de Smit. First we will fix again some notation. \\ \\
For a prime ideal $\mathfrak p \in \iok$, we denote by $\kappa(\mathfrak p)$ the finite residue field $\mathcal O _K / \mathfrak p$. The Frobenius endomorphism $\mathrm{Frob}_\mathfrak p$ of a $\kappa(\mathfrak p)$-algebra is defined by $x \mapsto x^{|\kappa(\mathfrak p)|}$. An endomorphism $f$ of a $\ok$-algebra $E$ is called a Frobenius lift (at $\mathfrak p$) if $f \otimes 1$ equals $\mathrm{Frob}_\mathfrak p$ on $E \otimes _{\mathcal O _K}\kappa(\mathfrak p)$. 
\begin{definition}
Let $E$ be a torsion-free $\ok$-algebra. A $\Lambda_K$-structure on $E$ is given by a family of endomorphisms $(f_\mathfrak p)$ indexed by the (non-zero) prime ideals $\mathfrak p$ of $\ok$, such that for all $\mathfrak p, \mathfrak q$, we have \begin{itemize}
\item[1)] $f_\mathfrak p \circ f_\mathfrak q = f_\mathfrak q \circ f_\mathfrak p$,
\item[2)] $f_\mathfrak p$ is a Frobenius lift.
\end{itemize}
\end{definition}

\begin{definition}
A $K$-algebra $E$ is said to have an \textit{integral} $\Lambda_K$-structure if there exists a $\ok$-algebra $\widetilde E$ with $\Lambda_K$-structure and an isomorphism $E \cong \widetilde E \otimes_{\ok} K$. In this case, $\widetilde E$ is called an integral model of $E$.
\end{definition}

\begin{remark}
The Frobenius-lift property is vacuous for $K$-algebras. This is why we need to ask for an integral structure.
\end{remark} \noindent
In \cite{BdS}, Borger and de Smit were able to classify finite, \'etale $K$-algebras with integral $\Lambda_K$-structure. Their result can be described as an arithmetic refinement of the classical \textit{Grothendieck-Galois correspondence}, which says that the category $\mathfrak E_K$ of finite, \'etale $K$-algebras is antiequivalent to the category $\mathfrak {S}_{G_K}$ of finite sets equipped with a continuous action of the absolute Galois group $G_K = \mathrm{Gal}(\overline K / K)$\footnote{The morphisms are given by $K$-algebra homomorphisms resp.  $G_K$-equivariant maps of sets.}. The equivalence is induced by the contravariant functor $A \mapsto {Hom}_{K\text{-}\mathrm{alg}}(A,\overline K)$. \\ 
The first observation is that giving a $\Lambda_K$-structure to a finite, \'etale $K$-algebra $E$ is the same as giving a monoid map\footnote{Recall that $\iok$ is generated as a (multiplicative) monoid by its (non-zero) prime ideals.} \eqst{\iok \to \mathrm{End}_{G_K}({Hom}_{K\text{-}\mathrm{alg}}(E,\overline K))\, ,} so that we end up with an action of the direct product $\iok \times G_K$ on $\mathrm{Hom}_{K\text{-}\mathrm{alg}}(E,\overline K)$. \\
Asking for an integral model of $E$ is much more delicate and is answered beautifuly in \cite{BdS} by making extensive use of class field theory as follows.

\begin{theorem}[\cite{BdS} Theorem 1.2] 
Let $E$ be a finite, \'etale $K$-algebra with $\Lambda_K$-structure. Then $E$ has an integral model if and only if there is an integral ideal $\mathfrak f \in \iok$ such that the action of $\iok \times G_K$ on $\mathrm{Hom}_{K\text{-}\mathrm{alg}}(E,\overline K)$ factors (necessarily uniquely) through the map $\iok \times \kab \longrightarrow DR_\mathfrak f$ given by the natural projection on the first factor and by the Artin reciprocity map\footnote{$G_K \to G_k^{ab}\to C_\mathfrak f \subset DR_\mathfrak f$\,.} on the second factor.
\end{theorem} \noindent
In particular one obtains the following arithmetic refinement of the classical Grothendieck-Galois correspondence.
\begin{corollary}[\cite{BdS}]
\label{ilggc}
The functor $\mathfrak H_K : E \mapsto \mathrm{Hom}_{K\text{-}\mathrm{alg}}(E,\overline K)$ induces a contravariant equivalence \eq{\label{functorBdS}\mathfrak H _K : \mathfrak E _{\Lambda,K} \longrightarrow \mathfrak S _{\mathrm{DR}_K}} between the category $\mathfrak E_{\Lambda,K}$ of finite, \'etale $K$-algebras with integral $\Lambda_K$-structure and the category $\mathfrak S_{\mathrm{DR}_K}$ of finite sets equipped with a continuous action of the Deligne-Ribet monoid $\mathrm{DR}_K$\footnote{The morphisms are given by $K$-algebra homomorphisms respecting the integral $\Lambda_K$-structure resp.\ by $\mathrm{DR}_K$-equivariant maps of finite sets.}. 
\end{corollary} \noindent
Note that we will use the same notation $\mathfrak H _K$ to denote the induced functor \eq{\label{profunctor} 
\mathfrak E_{ind\text{-}\Lambda,K} \longrightarrow \mathfrak S_{pro\text{-}\mathrm{DR}_K}} from the category of inductive systems in $\mathfrak E_{\Lambda,K}$ to projective systems in $\mathfrak S_{\mathrm{DR}_K}$.

\section{A simple decomposition of the Deligne-Ribet monoid}
\label{projsec}
\noindent
In this section, we describe an observation on the Deligne-Ribet monoid that will be used later on. First, notice (see (2.5) \cite{DelRib}) that for ideals $\mathfrak a, \mathfrak b$ and $\mathfrak d$ in $\iok$ we have the simple fact \eq{\label{deltacalculus}\mathfrak a \sim_\mathfrak f \mathfrak b \Leftrightarrow \mathfrak d \mathfrak a \sim_{\mathfrak d \mathfrak f} \mathfrak d \mathfrak b\,.}
This allows us to define a $\mathrm{DR}_K$-equivariant embedding \eqst{ \mathfrak d \cdot : DR_\mathfrak f \hookrightarrow DR_{\mathfrak d \mathfrak f} \ ; \ \mathfrak a \mapsto \mathfrak d \mathfrak a\, ,} and we can identify $DR_\mathfrak f$ with its image $\mathfrak d DR_{\mathfrak d \mathfrak f}$. 
Now taking projective limits, we obtain an injective map \eq{\label{delta}\varrho_\mathfrak d : \mathrm{DR}_K \to \mathrm{DR}_K} defined by \diagramnr{sequofmaps}{\varprojlim_\mathfrak f DR_\mathfrak f \ar[r]^{\cong}_{\mathfrak d \cdot} &\varprojlim_\mathfrak f \mathfrak d DR_{\mathfrak d \mathfrak f}\ar[r]_{inc} & \varprojlim_\mathfrak f DR_{\mathfrak d \mathfrak f}\ar[r]^{\cong} & \varprojlim_\mathfrak f DR_\mathfrak f, } which is in fact just a complicated way of writing the multiplication map
\eqst{\mathfrak a \in \mathrm{DR}_K \longmapsto \mathfrak d \mathfrak a \in \mathrm{DR}_K\,.}
We profit from our reformulation in that we see immediately that the image $\mathrm{Im}(\varrho_\mathfrak d) = \varprojlim_\mathfrak f \mathfrak d DR_{\mathfrak d \mathfrak f}$ is a closed subset of $\mathrm{DR}_K$. Also, using (\ref{deltacalculus}), we see that the complement of $\mathrm{Im}(\varrho_\mathfrak d)$ in $\mathrm{DR}_K$ is closed, and therefore we obtain, for every $\mathfrak d \in \iok$, a (topological) decomposition \eq{\label{decomposition} \mathrm{DR}_K = \mathrm{Im}(\varrho_\mathfrak d) \sqcup \mathrm{Im}(\varrho_\mathfrak d)^c\,.}

\section{The endomotive $\mathcal E _K$}
\label{E_K} \noindent
For every number field $K$, we want to construct an algebraic endomotive $\mathcal E _K$. \\ 
The main tool for this purpose is provided by the refined Grothendieck-Galois correspondence (see \ref{ilggc}) between finite, \'etale $K$-algebras with integral $\Lambda_K$-structure and finite sets with a continuous action of the Deligne-Ribet monoid $\mathrm{DR}_K$. \\
We observe that $\mathrm{DR}_K$ is acting continuously on the finite monoid $DR_\mathfrak f$, for every $\mathfrak f \in \iok$. Therefore, we know, by correspondence \ref{ilggc}, that there exists a finite, \'etale $K$-algebra $E_\mathfrak f$ with integral $\Lambda_K$-structure, such that \eqst{DR_\mathfrak f \cong \mathrm{Hom}_{K\text{-}\mathrm{alg}}(E_\mathfrak f,\overline K)\,.}
Further, using the decomposition $DR_\mathfrak f \cong \coprod_{\mathfrak d \mid \mathfrak f }C_{\mathfrak f / \mathfrak d}$ (cf., \eqref{prop2}), where $C_\mathfrak d \cong \mathrm{Gal}(K_\mathfrak d / K) = \mathrm{Hom}_{K\text{-}\mathrm{alg}}(K_\mathfrak d,\overline K)$, we see, again by invoking our correspondence \ref{ilggc}, that we have the following description of $E_\mathfrak f$. 
\begin{lemma}
For every  $\mathfrak f \in \iok$, we have an isomorphism of finite, \'etale $K$-algebras with integral $\Lambda_K$-structure \eq{\label{arithalg}E_\mathfrak f \cong \prod _{\mathfrak d \mid \mathfrak f} K_{\mathfrak f / \mathfrak d}\, ,} where $K_\mathfrak d \subset K^{ab}$ denotes the strict ray class field (of conductor $\mathfrak d$) in the maximal abelian extension $K^{ab}$ of $K$. In particular, we see that \eq{\label{abelian} DR_\mathfrak f \cong \mathrm{Hom}_{K\text{-}\mathrm{alg}}(E_\mathfrak f,K^{ab})\,.}
\end{lemma}\noindent
Further, the projection maps $\pi_{\mathfrak f, \mathfrak f'} : DR_{\mathfrak f'} \to DR_\mathfrak f$, for every $\mathfrak f \mid \mathfrak f'$, are equivariant with respect to the action of $\mathrm{DR}_K$. Therefore, we obtain, again by making use of our correspondence \ref{ilggc}, an inductive system $(E_\mathfrak f)_{\mathfrak f \in \iok}$ of finite, \'etale $K$-algebras with integral $\Lambda_K$-structure, with transition maps \eq{\label{transmaps}\xi_{\mathfrak f,\mathfrak f '} : E_\mathfrak f \to E_ {\mathfrak f'}} defined by \eqst{\xi_{\mathfrak f,\mathfrak f'} = \mathfrak H_K^{-1}(\pi_{\mathfrak f,\mathfrak f'})\,.}
\begin{definition}\label{defek}
We define the unital and commutative $K$-algebra $E_K$ to be the inductive limit of the system $(E_\mathfrak f)_{\mathfrak f \in \iok}$, i.e., we have \eqst{E_K = \varinjlim_\mathfrak f E_\mathfrak f\,.}
\end{definition}\noindent
If we denote the action of $\mathfrak d \in I_K$ on $E_\mathfrak f$ (coming from its integral $\Lambda_K$-structure) in terms of endomorphisms $f_\mathfrak d^\mathfrak f : E_\mathfrak f \to E_\mathfrak f$, we have, for all $\mathfrak f, \mathfrak f', \mathfrak d \in \iok$ with $\mathfrak f \mid \mathfrak f'$, the following compatibility relation: \diagramnr{diagcomp}{E_\mathfrak f \ar[r]^{\xi_{\mathfrak f,\mathfrak f'}} \ar[d]^{f_\mathfrak d^\mathfrak f} &E_{\mathfrak f '} \ar[d]^{f_\mathfrak d^{\mathfrak f '}} \\ E_\mathfrak f \ar[r]^{\xi_{\mathfrak f,\mathfrak f'}} & E_{\mathfrak f '}\,.}
In particular, this means that we obtain a natural action of $\iok$ on the inductive limit $E_K$, given, for every $\mathfrak d \in \iok,$ by the $K$-algebra endomorphism $\sigma_\mathfrak d$, defined by \eqst{\sigma_\mathfrak d = \varinjlim _\mathfrak f f_\mathfrak d ^\mathfrak f : E_K \to E_K\,.} Recall that $\varrho_\mathfrak d : \mathrm{DR}_K \to \mathrm{DR}_K$ is defined to be the injective multiplication-by-$\mathfrak d$ map (see \eqref{delta}). 
\begin{lemma}
For every $\mathfrak d \in \iok$, we have the equality \eqst{\varrho_\mathfrak d = \mathfrak H _K(\sigma_\mathfrak d)\, ,} i.e., under the contravariant equivalence of categories (cf., corollary \ref{ilggc}), the two maps $\sigma_\mathfrak d$ and $\varrho_\mathfrak d$ correspond to each other. 
\end{lemma}
\begin{proof}
This follows by applying the functor $\mathfrak H _K$ to the diagram \eqref{diagcomp} and the fact that the $\Lambda_K$-structure of $E_\mathfrak f$ does come from action of $DR_K$ on $DR_\mathfrak f$ given by multiplication.
\end{proof}\noindent
Now, by making use of the decomposition $\mathrm{DR}_K = \mathrm{Im}(\varrho_\mathfrak d) \sqcup \mathrm{Im}(\varrho_\mathfrak d)^c$ (see \eqref{decomposition}), we conclude the following decompositions of our algebra $E_K$.
\begin{lemma} For every $\mathfrak d \in \iok$, there exists a projection\footnote{This means $\pi_\mathfrak d ^2 = \pi_\mathfrak d$\,.} $\pi_\mathfrak d \in E_K$, such that \eqst{\mathrm{Im}(\varrho_\mathfrak d) \cong \mathrm{Hom}_{K\text{-}\mathrm{alg}}(\pi_\mathfrak d E_K,\overline K)}
and 
\eqst{\mathrm{Im}(\varrho_\mathfrak d)^c \cong \mathrm{Hom}_{K\text{-}\mathrm{alg}}((1-\pi_\mathfrak d) E_K,\overline K)\,.} In particular, for every $\mathfrak d \in \iok$, we obtain the following decomposition \eqst{E_K = \pi_\mathfrak d E_K \oplus (1-\pi_\mathfrak d)E_K\,.}
\end{lemma}
\begin{proof} 
This follows directly from the correspondence \ref{ilggc} and the decomposition \eqref{decomposition}. Note that for two unital and commutative $K$-algebras $A$ and $B$, we have the elementary decomposition\footnote{Observe that every homomorphism $f\in \mathrm{Hom}_{K\text{-}\mathrm{alg}}(A\oplus B,\overline K)$ sends necessarily exactly one of the two elements $(1_A,0), (0,1_B) \in A\oplus B$ to zero.} $\mathrm{Hom}_{K\text{-}\mathrm{alg}}(A\oplus B,\overline K) \cong \mathrm{Hom}_{K\text{-}\mathrm{alg}}(A,\overline K) \sqcup \mathrm{Hom}_{K\text{-}\mathrm{alg}}(B,\overline K)$. 
\end{proof}\noindent
We can now make the following crucial definition (again by using the correspondence \ref{ilggc}).
\begin{definition} For every $\mathfrak d \in \iok$, we define the endomorphism $\rho_\mathfrak d \in \mathrm{End}_{K\text{-}\mathrm{alg}}(E_K)$ by \eq{\label{rhodef}\rho_\mathfrak d = i \circ \mathfrak H_K^{-1}(\varrho_\mathfrak d^{-1} : \mathrm{Im}(\varrho_\mathfrak d) \overset \cong \longrightarrow \mathrm{DR}_K)\, ,}
where $i : \pi_\mathfrak d E_K \to E_K$ denotes the natural inclusion. 
\end{definition}\noindent
In order to see that the endomorphisms $\rho_\mathfrak d$ are actually well-defined, one has to observe that all the maps occurring in \eqref{sequofmaps} are $\mathrm{DR}_K$-equivariant, and therefore our use of $\mathfrak H _K^{-1}$ is justified. 
\begin{remark}
The reader should be aware of the fact that the $\rho_\mathfrak d$ are not level preserving like the $\sigma_\mathfrak d$, in the sense that the latter restrict to maps $E_\mathfrak f \to E_\mathfrak f$, for every $\mathfrak f \in \iok$. 
\end{remark} \noindent
Let us give a schematic overview of what we have done so far. \begin{proposition}\label{schoverview} For every $\mathfrak d \in \iok$, we have the following everywhere commutative diagram
\diagramnr{basicrhosigma}{ E_K \ar[rd]^{\sigma_\mathfrak d} \ar[d]^{pr} &  & E_K  \\ \pi_\mathfrak d E_K \ar@/_1.5pc/[rr]_{id} \ar[r]^\cong & E_K \ar[ur]^{\rho_\mathfrak d} \ar[r]^\cong & \pi_\mathfrak dE_K \ar[u]^{i}\,.}
\end{proposition}
\begin{proof}
This is only a translation of the results from section \ref{projsec} in terms of the correspondence described in corollary \ref{ilggc}.
\end{proof}
\noindent
The following relations hold by construction.
\begin{lemma}
\label{easylemma}
For all $\mathfrak d, \mathfrak e$ in $\iok$ and every $x \in E_K$, we have \begin{equation*}\begin{array}{ccc}  \rho_\mathfrak d(1) = \pi_\mathfrak d ,& \pi_\mathfrak d \pi_{\mathfrak e} = \pi_{lcm(\mathfrak d,\mathfrak e)}, \\ \sigma_\mathfrak d \circ \sigma_\mathfrak e = \sigma_{\mathfrak d \mathfrak e}, & \rho_\mathfrak d \circ \rho_\mathfrak e = \rho_{\mathfrak d \mathfrak e}, \\  \rho_\mathfrak d \circ \sigma_\mathfrak d(x) = \pi_\mathfrak d x ,& \sigma_\mathfrak d \circ \rho_\mathfrak d (x) = x . \end{array}\end{equation*}
\end{lemma} 
\begin{proof}
For the second assertion on the first line, notice that, by remark \ref{intersection}, we have \eqst{\mathfrak d \cdot \mathrm{DR}_K \cap \mathfrak e \cdot \mathrm{DR}_K = lcm(\mathfrak d,\mathfrak e) \cdot \mathrm{DR}_K\,.}
\end{proof}\noindent
Finally, we can define our desired algebraic endomotive. 
\begin{proposition}
\label{defalgek}
The inductive system $(E_\mathfrak f)_{\mathfrak f \in I_K}$,with transition maps $(\xi_{\mathfrak f,\mathfrak f '})$ defined in \eqref{transmaps}, of finite, \'etale $K$-algebras, together with the action of $I_K$ on $E_K = \varinjlim _\mathfrak f E_\mathfrak f$ in terms of the $\rho_\mathfrak d$, defines an algebraic endomotive $\mathcal E _K$ over $K$, given by \eqst{\mathcal E _K = E_K \rtimes I_K\,.}
\end{proposition}
\begin{proof}
We only have to show that the $\rho_\mathfrak d$ induce an isomorphism between $E_K$ and $\pi_\mathfrak d E_K$, but this is contained in proposition \ref{schoverview}.
\end{proof}
\begin{remark}
As done in \cite{fun} for the case of $K=\q$, it is possible to construct integral models of our algebraic endomotives $\mathcal E _K$. These integral models will play a role in a forthcoming paper.
\end{remark}

\section{Proof of Theorem \ref{ThmA} and \ref{ThmB}}

\begin{theorem}
The algebraic endomotive $\mathcal E _K$ gives rise to a $C^*$-dynamical system that is naturally isomorphic to the BC-system $\mathcal A _K$ (see \eqref{bctype}).
\end{theorem} \noindent
We will prove the theorem in two steps.
\subsection{Step One}
\noindent
For every number field $K$ there is a natural map of topological monoids \eqst{\Psi : Y_K = \okhat \times_{\okhat^\times} \kab \longrightarrow \mathrm{DR}_K} given by \eqst{[\rho,\alpha] \longmapsto \iota(\rho)\alpha^{-1}\,.}
This map is well defined due to the fact that $\iota(s) = [s]^{-1} \in \mathrm{Gal}(K^{ab}/K)$ for $s \in \okhat^\times$ (see proposition \ref{centralprop}). 
\begin{proposition}
\label{keyprop}
The map $\Psi$ is an equivariant isomorphism of topological monoids with respect to the natural actions of each of $\iok$ and $\kab$.  
\end{proposition}
\begin{proof}
It is enough to show that the map \eq{\label{fundisom} \Psi_\mathfrak f : \mathcal O_K / \mathfrak f \times _{(\ok / \mathfrak f)^\times} C_\mathfrak f \longmapsto DR_{\mathfrak f}\, ,} given by \eqst{[\rho,\alpha] \mapsto \iota_\mathfrak f(\rho) \alpha^{-1}\, ,} is an isomorphism of finite monoids for every $\mathfrak f \in \iok$. This follows from the compactness of $Y_{K,\mathfrak f} = \ok / \mathfrak f \times_{(\ok / \mathfrak f)^\times} C_\mathfrak f$ and the simple fact that $\varprojlim_\mathfrak f Y_{K,\mathfrak f} \cong Y_K$. Denote by $\pi_0$ the group of connected components of the infinite idele group $(\mathbb A_{K,\infty})^\times$ and consider, for every $\mathfrak f \in \iok$, the following everywhere commutative and exact diagram \diagramnr{basicnt}{ & & 1 & \\ \pi_0\times(\ok / \mathfrak f)^\times \ar[r] & C_\mathfrak f \ar[r] & C_K \ar[u] \ar[r] & 1 \, \, \\ (\ok / \mathfrak f)^\times \ar[u] \ar[r]^{j_\mathfrak f} & C_\mathfrak f \ar[r] \ar[u]_{=} & C_1 \ar[u] \ar[r] & 1 \, ,  \\ & & \pi_0 \ar[u] &}
as can be found for example in \cite{nk}. From \ref{eulertotient} we know that $\ok / \mathfrak f$ and $\coprod_{\mathfrak d \mid \mathfrak f} (\ok / \mathfrak d)^\times$ are isomorphic as sets, but they are in fact isomorphic as monoids:  
\begin{lemma} There is an isomorphism of monoids $\sigma_\mathfrak f : \ok / \mathfrak f \to \coprod _{\mathfrak d | \mathfrak f} (\ok / \mathfrak d)^\times$ such that the following diagram is commutative \diagram{\ok / \mathfrak f \ar[d]_{\sigma_\mathfrak f} \ar[rr]^{\iota_\mathfrak f} & & DR_\mathfrak f \ar[d] \\ \coprod_{\mathfrak d \mid \mathfrak f} (\ok / (\mathfrak f  / \mathfrak d))^\times \ar[rr]^{\coprod j_{\mathfrak f / \mathfrak d}} & & \coprod_{\mathfrak d \mid \mathfrak f} C_ {\mathfrak f / \mathfrak d}\,.}
\end{lemma}
\begin{proof}
It is enough to consider the case $\mathfrak f = \mathfrak p ^k$ where $\mathfrak p$ a prime ideal. The general case follows using the chinese remainder theorem. It is well known that $\ok / \mathfrak p ^k$ is a local ring with maximal ideal $\mathfrak p / \mathfrak p ^k$, i.e., we have a disjoint union $\ok / \mathfrak p ^k = (\ok / \mathfrak p ^k)^\times \sqcup \mathfrak p / \mathfrak p^k$. Further, there is a filtration $\{0\} \subset \mathfrak p ^{k-1} / \mathfrak p^{k} \subset \mathfrak p ^{k-2}/\mathfrak p ^k \subset \ldots \subset \mathfrak p / \mathfrak p^k$ and, for $x \in \mathfrak p / \mathfrak p^k$ and $x_+ \in \ok$ a (positive) lift, we have \eqst{x \in \mathfrak p^{k-i} / \mathfrak p ^k - \mathfrak p ^{k-i+1} / \mathfrak p ^k \Leftrightarrow \mathfrak p^{k-i} \ || \ (x_+) \Leftrightarrow  x_+ \in (\ok / \mathfrak p ^{k-i+1})^\times \, .}
A counting argument as in \ref{eulertotient}, and recalling the definition of \eqref{prop2}, finishes the proof.
\end{proof} \noindent
Now, we can conclude the \textbf{injectivity} of $\Psi_\mathfrak f$, because assuming $\iota_\mathfrak f(\rho)\alpha^{-1} = \iota_\mathfrak f(\sigma)\beta^{-1}$, for $\rho,\sigma \in \ok / \mathfrak f$, $\alpha,\beta \in C_\mathfrak f$, we must have that $\alpha$ and $\beta$ map to the same element in $C_1$. This is, because $\iota_\mathfrak f(\sigma)\alpha\beta^{-1}$ lies in the image of $\iota_\mathfrak f$ and is therefore mapped to the trivial element in $C_1$. But lying over the same element in $C_1$ means that there exists $ s \in (\ok / \mathfrak f)^\times$ such that $\alpha \beta^{-1} = [s] = \iota_\mathfrak f(s)^{-1}$, and therefore, we get $[\rho,\alpha] = [\sigma,\beta] \in Y_{K,\mathfrak f}$. \\
To prove \textbf{surjectivity}, we use again the decomposition $DR_\mathfrak f = \coprod_{\mathfrak d | \mathfrak f} C_{\mathfrak f / \mathfrak d}$. We have to show that, for every $\mathfrak d \mid \mathfrak f$, we have $C_\mathfrak f \cdot \mathrm{Im}(j_\mathfrak d) = C_{\mathfrak f / \mathfrak d}$, where $\cdot$ denotes the multiplication in the monoid $DR_\mathfrak f$. One has to be careful because it is not true that $C_\mathfrak f$ acts transitively on $\mathrm{Im}(j_\mathfrak d)$.\footnote{Consider for the example the case when $gcd(\mathfrak d,\mathfrak f / \mathfrak d)=1$, then $\mathfrak d, \mathfrak d^2 \in \mathrm{Im}(j_\mathfrak d)$ but $\mathfrak d ^2 \notin C_\mathfrak f \cdot \mathfrak d$} Instead, we show that $C_\mathfrak f \cdot \mathfrak d$ intersects every fibre of $C_{\mathfrak f / \mathfrak d} \to C_1$ non-trivially. For every element $x \in C_1$, we find lifts $x_\mathfrak f \in C_\mathfrak f$ and $x_{\mathfrak f / \mathfrak d} \in C_{\mathfrak f / \mathfrak d}$ such that $x_\mathfrak f$ is mapped to $x_{\mathfrak f / \mathfrak d}$ under the natural projection $DR_\mathfrak f \to DR_{\mathfrak f / \mathfrak d}$. Our claim is equivalent to $x_\mathfrak f \mathfrak d \sim_\mathfrak f x_{\mathfrak f /\mathfrak d}\mathfrak d$, which is equivalent (see \eqref{deltacalculus}) to $x_\mathfrak f \sim_{\mathfrak f /\mathfrak d} x_{\mathfrak f / \mathfrak d}$, which is is true by construction. \\
To finish the proof, we have to show that $\Psi$ is compatible with each of the natural actions of $\iok$ and $\kab$ on $Y_K$ and $\mathrm{DR}_K$ respectively. Let us recall that the action of $\iok \cong \okhat^\natural / \okhat ^\times$ on $Y_K$ is given by $s[\rho,\alpha] = [\rho s,[s]^{-1}\alpha]$, and $\kab$ is acting by $\gamma [\rho,\alpha] = [\rho,\gamma\alpha]$. The equivariance of $\Psi$ under the action of $\kab$ is clear, and the equivariance under the action of $\iok$ follows from proposition \ref{centralprop}, namely $\Psi(s[\rho,\alpha]) = \iota(\rho)\iota(s)[s]\alpha^{-1}\overset{\ref{centralprop}}= \iota(\rho)(s)\alpha^{-1}= (s)\Psi([\rho,\alpha])$. This shows that $\Psi$ is an isomorphism of topological $\mathrm{DR}_K$-monoids.
\end{proof}
\noindent
Now we obtain immediately:
\begin{corollary}
Let $K$ be a number field. Then the isomorphism $\Psi$ from above induces an isomorphism \eqst{\Psi : A_K = C(Y_K)\rtimes \iok 
\longrightarrow \mathcal E _K^{an} = C(\mathrm{DR}_K)\rtimes \iok} between the $C^*$-algebra $A_K$ of the BC-system $\mathcal A _K$ and the analytic endomotive $\mathcal E _K^{an}$. 
\end{corollary}

\subsection{Step Two}
\noindent
It remains to show that $\mathcal E _K$ defines a measured analytic endomotive whose time evolution on $\mathcal E _K ^{an}$ agrees with the time evolution of the BC-system $\mathcal A _K$ (see \eqref{bctype}). \\ \\
First, we will show that $\mathcal E _K$ is a uniform endomotive, i.e., the normalized counting measures $\mu_\mathfrak f$ on $DR_\mathfrak f$ give rise to a measure $\mu _K = \varprojlim \mu_\mathfrak f$ on $\mathrm{DR}_K = \mathrm{Hom}(E_K,\overline K)$. \\
Then, in order to show that $\mu_K$ indeed defines a time evolution on $\mathcal E _K ^{an}$ using the procedure described in section \ref{ttcon} which, in addition, agrees with the time evolution of $\mathcal A _K$, we only have to show that $\mu_K$ equals the measure $\mu_1$ on $Y_K$ characterizing the unique $KMS_1$-state of $\mathcal A _K$ (see section \ref{measure1}). \\
This follows from standard arguments in Tomita-Takesaki theory. Namely, if $\mu_K$ defines a time evolution $\sigma_t$ on $\mathcal E _K ^{an}$, then we know a priori that the corresponding state $\varphi_{\mu_K} : \mathcal E _K^{an}\to \complex$ is a $KMS_1$-state characterizing the time evolution $\sigma_t$ uniquely (cf., chapter 4, section 4.1 \cite{ConMar} and the references therein).

\begin{lemma} Let $\mathfrak f$ be an arbitrary ideal in $\iok$. Then we have \eqst{|DR_\mathfrak f |=  2^{r_1} h_K N _{K/\q}(\mathfrak f)\, ,}
where $h_K$ denotes the class number of $K$ and $r_1$ is equal to the number of real embeddings of $K$.
\end{lemma}
\begin{proof}
Recall the fundamental exact sequence of groups (see, e.g., \cite{nk}) \diagram{1 \ar[r] & U_\mathfrak f \ar[r] &  (\ok / \mathfrak f)^\times \ar[r]^{ \ \ \ j_\mathfrak f} & C_\mathfrak f \ar[r] & C_1 \ar[r] & 1 \, ,} with notations as in \eqref{basicnt} and $U_\mathfrak f$ making the sequence exact, from which we obtain immediately \eqst{ |C_\mathfrak f| = \frac{2^{r_1}\varphi_K(\mathfrak f) h_K}{|U_\mathfrak f |} \, ,}
where $\varphi_K$ denotes the generalized Euler totient function from Appendix \ref{appeuler}.
In order to count the elements of $\mathrm{DR}_K$ we notice (cf., proposition \ref{iotakernel}) that the fibers of the natural projection $\ok / \mathfrak f \times C_\mathfrak f \to \ok/\mathfrak f \times_{(\ok / \mathfrak f)^\times} C_\mathfrak f \cong DR_\mathfrak f$ all have the same cardinality given by $\frac{\varphi_K(\mathfrak f)}{ |U_\mathfrak f |}$ and this finishes the proof.
\end{proof}

\begin{lemma}
\label{fibercount}
Let $\mathfrak f$ and $\mathfrak g$ be in $\iok$ such that $\mathfrak f$ divides $\mathfrak g$. Then the cardinalities of all the fibres of the natural projection $DR_\mathfrak g \to DR_\mathfrak f$ are equal to $|DR_\mathfrak g| /  |DR_\mathfrak f| = N _{K/\q} (\mathfrak g / \mathfrak f )$.
\end{lemma}
\begin{proof}
To show that all the cardinalities of the fibers of the projection $DR_\mathfrak g \to DR_\mathfrak f$ are equal, we look at the following commutative diagram (with the obvious maps) \diagram{\ok /\mathfrak g \times C_\mathfrak g \ar[r] \ar[d] & \ok /\mathfrak g \times_{(\ok / \mathfrak g)^\times} C_\mathfrak g \ar[d]^\xi \\ \ok/\mathfrak f \times C_\mathfrak f \ar[r] & \ok /\mathfrak f \times_{(\ok / \mathfrak f)^\times} C_\mathfrak f \, .}
All the maps in the diagram are surjective, and in order to show that the cardinalities of all the fibers of $\xi$ are equal, it is enough to show this property for the other three maps. In the proof of the preceding lemma, we have shown that the horizontal maps have this property, and for the remaining vertical map on the left, this property is trivial. Therefore, we conclude that the cardinalities of all the fibers of $\xi$ are equal and, together with the isomorphism \eqref{fundisom} and the preceding lemma, the assertion follows. 
\end{proof}

\begin{corollary}
The algebraic endomotive $\mathcal E _K$ is uniform.
\end{corollary}
\begin{proof}
Let $\mathfrak f,\mathfrak g \in \iok$ with $\mathfrak f  \mid   \mathfrak g$, and denote by $\xi$ the natural projection $DR_\mathfrak g \to DR_\mathfrak f$. In order to show that $\mathcal E _K$ is uniform, we have to show that $\xi_*(\mu_\mathfrak g) = \mu_\mathfrak f$, which follows directly from the preceding lemma. More precisely, if we take a subset $X \subset \mathrm{DR}_K$, we obtain \eqst{\xi_*(\mu_\mathfrak g) (X) = \mu_\mathfrak g (\xi^{-1}(X)) \overset{\ref{fibercount}}= |X| \cdot N _{K/\q}(\mathfrak g / \mathfrak f) / |DR_\mathfrak g| \overset{\ref{fibercount}}= |X| / |DR_\mathfrak f| = \mu_\mathfrak f(X)\,.}
\end{proof}

\begin{lemma}
Denote by $\widetilde \mu_\mathfrak f$ the push-forward of $\mu_1$ under the projection $\pi_\mathfrak f :Y_K \overset\Psi\longrightarrow \mathrm{DR}_K \longrightarrow DR_{\mathfrak f}$. Then $\widetilde \mu_\mathfrak f$ is the normalized counting measure on $DR_\mathfrak f$. 
\end{lemma}
\begin{proof}
We only have to show that \eqst{\widetilde \mu_\mathfrak f (q) = \widetilde \mu_\mathfrak f (q') \text{ for all }q,q' \in DR_\mathfrak f\, ,} because by definition we have $1 = \widetilde \mu_\mathfrak f (DR_\mathfrak f) = \sum_q \widetilde \mu_\mathfrak f(q)$. Recall that $\mu_1$ is defined to be the push forward of the product measure $\mu = \prod_\mathfrak p \mu_\mathfrak p \times \mu_\mathcal G$ on $\okhat \times \kab$, where the $\mu_\mathfrak p$ and $\mu_\mathcal G$ are normalized Haar measures under the natural projection $\pi : \okhat \times \kab \to Y_K$ (cf., section \ref{measure1}). It is immediate that for given $q$ and $q'$ in $\iok$, we find $m = m_{q,q'} \in \iok$ and $s = s_{q,q'} \in \kab$, such that the translate of $X _q = \pi_\mathfrak f^{-1}(\pi^{-1}(q))$ under $m$ and $s$ equals $X_{q'}$, i.e., \eqst{m X_q  s := \{ (m+\rho,s\alpha) \ | \ (\rho,\alpha) \in X_q \} = X_{q'}\,.} 
Due to translation invariance of Haar measures we can conclude $\mu(X_q)=\mu(mX_qs)=\mu(X_{q'})$ and therefore \eqst{\widetilde \mu_\mathfrak f(q)=\widetilde \mu_\mathfrak f(q')\,.}
\end{proof}

\begin{lemma}
The measure $\mu_K = \varprojlim \mu _\mathfrak f$ satisfies the scaling condition \eqref{scalingcond}.
\end{lemma}
\begin{proof}

Let $\mathfrak d$ and $\mathfrak f$ be in $\iok$. Without loss of generality, we can assume that $\mathfrak d$ divides $\mathfrak f$, because we are looking at the limit measure. Recall further the commutative diagram \diagram{DR_\mathfrak f \ar[r]^{\mathfrak d \cdot} \ar@{>>}[d] & DR_\mathfrak f \, .\\DR_{\mathfrak f / \mathfrak d} \ar@{^{(}->}[ur]^{\mathfrak d\cdot }} 
\noindent
In order to show that $\mu_K$ satisfies the scaling condition, it is enough to show that the cardinalities of the (non-trivial) fibers of the multiplication map $\mathfrak d \cdot : DR_\mathfrak f \to DR_\mathfrak f$ are all equal to the norm $N_{K/\q} (\mathfrak d) = |\ok / \mathfrak d|$. By the commutativity of the last diagram, we only have to show that the fibres of the natural projection $DR_\mathfrak f \to DR_{\mathfrak f / \mathfrak d}$ all have cardinality $ N _{K/\q}(\mathfrak d)$. This follows immediately from lemma \ref{fibercount}. 
\end{proof}
\noindent
As corollary of the preceding two lemma we obtain the following.
\begin{proposition}
\label{equalmeasures}
We have the equality of measures \eqst{\mu_K = \mu_1\,.}
\end{proposition}
\begin{proof}
We have seen that $\mu _K$ satisfies the two defining properties of $\mu_1$ (cf., section \ref{measure1}).
\end{proof}

\begin{corollary}
The procedure described in \ref{ttcon} defines a time evolution $(\sigma_t)_{t\in \real}$ on $\mathcal E _K^{an}$, and the resulting measured analytic endomotive $\mathcal E ^{mean} = (\mathcal E_K^{an},(\sigma_t)_{t \in \real})$ is naturally isomorphic to $\mathcal A _K$, via $\Psi$.
\end{corollary}
\noindent
Next we will show that $\mathcal E _K=E_K \rtimes \iok$ provides $\mathcal A _K$ with an arithmetic subalgebra.
This follows in fact directly from the construction.

\begin{theorem}
For all number fields $K$ the BC-systems $\mathcal A _K$ (resp. $\mathcal E _K^{mean}$) posses an arithmetic model with arithmetic subalgebra given by the algebraic endomotive $\mathcal E _K = E_K \rtimes I_K$. 
\end{theorem}
\begin{proof}
Recall from section \ref{kmsinfty} that extremal $KMS_\infty$-states are indexed by $\kab \overset{\eqref{prop3}}\cong \mathrm{DR}_K^\times \subset \mathrm{Hom}_{K\text{-}\mathrm{alg}}(E_K, \overline K)$, i.e., an extremal $KMS_\infty$-state $\varrho_\omega$, for $\omega \in \mathrm{DR}_K^\times$, is given on a function $f \in C(\mathrm{DR}_K)$ simply by \eqst{\varrho_\omega (f) = f(\omega)\,.}
Now, if we take an element $\mathrm{ev}_a \in E_K \subset C(\mathrm{DR}_K)$, which was defined by $\mathrm{ev}_a : g \in \mathrm{Hom}_{K\text{-}\mathrm{alg}}(E_K,\overline K) = \mathrm{Hom}_{K\text{-}\mathrm{alg}}(E_K,K^{ab}) \mapsto g(a) \in K^{ab}$ (see \eqref{abelian}), we find that \eqst{\varrho_\omega(\mathrm{ev}_a) = \mathrm{ev}_a(\omega)=\omega(a) \in K^{ab}\, ,} and this shows, together with the definition of $E_K$, that property (v) from the list of axioms of a Bost-Connes system is valid. In order to show property (vi), we take a symmetry $\nu \in \kab$ and simply calculate \eqst{{}^{\nu}\hspace{-0.6mm}\varrho_\omega(\mathrm{ev}_a) = \varrho_\omega({}^{\nu}\hspace{-0.6mm} \mathrm{ev}_a)={}^{\nu}\hspace{-0.6mm}\mathrm{ev}_a(\omega)=\mathrm{ev}_a(\nu^{-1}\circ\omega) = \nu^{-1}(\omega(a))=\nu^{-1}(\varrho_\omega(\mathrm{ev}_a))\,.}
\end{proof}

\section{Outlook} \noindent
We would like to state some questions and problems which might be interesting for further research. 
\begin{itemize}
\item It might be interesting to use integral models $A_{\ok}$ of our BC-systems $\mathcal A _K$ (by using integral models of our arithmetic subalgebras), as done in \cite{fun} in the case of the classical BC-system for $K=\q$, to investigate whether general BC-systems can be defined over $\mathbb F _1$ or some (finite) extensions of $\mathbb F_1$.\smallskip
\item In a recent preprint \cite{CCpadic} Connes and Consani construct $p$-adic representations of the classical Bost-Connes system $\mathcal A _\q$ using its integral model $A_\ganz$. One of the main tools is thereby the classical Witt functor which attaches to a ring its ring of Witt vectors. Borger \cite{BorgWitt} has introduced a more general framework of Witt functors which are compatible with our arithmetic subalgebras. It might be interesting to construct analogous $p$-adic representations of general Bost-Connes systems. \smallskip
\item In particular, Connes and Consani \cite{CCpadic} recover $p$-adic $L$-functions in the p-adic representations of $\mathcal A _\q$. Using the results of \cite{DelRib}, it would be interesting to try to recover $p$-adic $L$-functions of totally real number fields in the $p$-adic representations of BC-systems of totally real number fields. \smallskip
\item On the other hand, it seems interesting to ask whether p-adic BC-systems are related to Lubin-Tate theory\footnote{In a forthcoming paper we will deal with these questions.}. \smallskip
\end{itemize}

\bigskip

\section{On uniqueness of arithmetic models. Appendix by Sergey Neshveyev}
\label{sergey} \noindent
The goal of this appendix is to show that the endomotive $\mathcal E_K$ constructed in the paper is, in an appropriate sense, the unique endomotive that provides an arithmetic model for the BC-system $\mathcal A_K$. We will also give an alternative proof of the existence of $\mathcal E_K$.\\ \\
Assume $\mathcal E=E\rtimes S$ is an algebraic endomotive such that the analytic endomotive $\mathcal E^{an}$ is $A_K=C(Y_K)\rtimes I_K$. By this we mean that $S=I_K$ and there exists a $Gal(\overline K/K)$- and $I_K$-equivariant homeomorphism of $\mathrm{Hom}_{K\text{-alg}}(E,\overline K)$ onto $Y_K=\widehat\OO_K\times_{\ohs_K}Gal(K^{ab}/K)$. Then $E$ considered as a $K$-subalgebra of $C(Y_K)$ has the following properties:
\begin{itemize}
\item[(a)] every function in $E$ is locally constant;
\item[(b)] $E$ separates points of $Y_K$;
\item[(c)] $E$ contains the idempotents $\rho^n_\aaa(1)$ for all $\aaa\in I_K$ and $n\in\N$;
\item[(d)] for every $f\in E$ we have $f(Y_K)\subset K^{ab}$ and the map $f\colon Y_K\to K^{ab}$ is $Gal(K^{ab}/K)$-equivariant.
\end{itemize}
Recall that the endomorphism $\rho_\aaa$ is defined by $\rho_\aaa(f)=f(\aaa^{-1}\cdot)$, with the convention that $\rho_\aaa(f)(y)=0$ if $y\notin \aaa Y_K$.

\begin{theorem}
\label{SergeyA}
The subalgebra $E_K=\varinjlim E_\mathfrak f$ of $C(Y_K)$ constructed in the paper is the unique $K$-subalgebra of~$C(Y_K)$ with properties (a)-(d). It is, therefore, the $K$-algebra of all locally constant $\kab$-valued $Gal(\kab/K)$-equivariant functions on $Y_K$.
\end{theorem}

\bp We have to show that if a $K$-subalgebra $E\subset C(Y_K)$ satisfies properties (a)-(d), then it contains every locally constant $K^{ab}$-valued $Gal(K^{ab}/K)$-equivariant function $f$.\\
Fix a point $y\in Y_K$. Let $L\subset \kab$ be the field of elements fixed by the stabilizer $G_y$ of $y$ in $Gal(K^{ab}/K)$. Then $f(y)\in L$ by equivariance.
\begin{lemma} \label{lgen}
The map $E\ni h\mapsto h(y)\in L$ is surjective.
\end{lemma}

\bp Let $L'$ be the image of $E$ under the map $h\mapsto h(y)$. Since $E$ is a $K$-algebra, $L'$ is a subfield of $L$. If $L'\ne L$ then there exists a nontrivial element of $Gal(L/L')\subset Gal(L/K)=Gal(K^{ab}/K)/G_y$.
Lift this element to an element $g$ of $Gal(K^{ab}/K)$. Then, on the one hand, $gy\ne\nolinebreak y$, and, on the other hand, for every $h\in E$ we have $h(gy)=gh(y)=h(y)$. This contradicts property~(b).
\ep

Therefore there exists $h\in E$ such that $h(y)=f(y)$. Since the functions $f$ and $h$ are locally constant, there exists a neighbourhood $W$ of $y$ such that $f$ and $h$ coincide on $W$. We may assume that $W$ is the image of an open set of the form
$$
\left(\prod_{v\in F}W_v\times\widehat\OO_{K,F}\right)\times W'\subset\widehat\OO_K\times Gal(K^{ab}/K)
$$
in $Y_K$, where $F$ is a finite set of finite places of $K$; here we use the notation $\widehat\OO_K=\prod_{v\in V_{K,f}}\OO_{K,v}$, $\widehat\OO_{K,F}=\prod_{v\in V_{K,f}\setminus F}\OO_{K,v}$. Furthermore, we may assume that $F=F'\sqcup F''$ and for $v\in F'$ we have $W_v\subset\pp_v^{n_v}\OO^\times_{K,v}$, while for $v\in F''$ we have $W_v=\pp_v^{n_v}\OO_{K,v}$. Since the functions $f$ and $h$ are equivariant, they coincide on the set $U=Gal(K^{ab}/K)W$. The equality
$$
Gal(K^{ab}/K)W=\left(\prod_{v\in F'}\pp_v^{n_v}\OO^\times_{K,v}\times\prod_{v\in F''}\pp_v^{n_v}\OO_{K,v}\times\widehat\OO_{K,F}\right)\times_{\ohs_K} Gal(K^{ab}/K)
$$
shows that the characteristic function $p$ of $U$ belongs to $E$: it is the product of $\rho_{\pp_v}^{n_v}(1)-\rho_{\pp_v}^{n_v+1}(1)$, $v\in F'$, and $\rho_{\pp_v}^{n_v}(1)$, $v\in F''$. Therefore $fp=hp\in E$.\\
Thus we have proved that for every point $y\in Y_K$ there exists a neighbourhood $U$ of $y$ such that the characteristic function $p$ of $U$ belongs to $E$ and $fp\in E$. By compactness we conclude that $f\in E$.
\ep

The following consequence of the above theorem shows that the arithmetic subalgebra $\mathcal E_K=E_K\rtimes I_K$ of the BC-system is unique within a class of algebras not necessarily arising from endomotives.

\begin{theorem}
\label{SergeyB}
The $K$-subalgebra $\mathcal E_K$ of $A_K$ constructed in the paper is the unique arithmetic subalgebra that is generated by some locally constant functions on $Y_K$ and by the elements $U_\aaa$ and $U_\aaa^*$, $\aaa\in I_K$.
\end{theorem}

\bp Assume $\mathcal E$ is such an arithmetic subalgebra. Consider the $K$-algebra $E=\mathcal E\cap C(Y_K)$. It satisfies properties (a)-(c), while (d) a priori holds only on the subset $Y_K^\times\subset Y_K$. However, the algebra $E$ is invariant under the endomorphisms $\sigma_\aaa$, $\aaa\in I_K$, defined by $\sigma_\aaa(f)=f(\aaa\,\cdot)=U_\aaa^*fU_\aaa$. Hence property (d) holds on the subsets $\aaa Y_K^\times$ of $Y_K$. Since $\cup_{\aaa\in I_K}\aaa Y_K^\times$ is dense in $Y_K$ and the functions in $E$ are locally constant, it follows that (d) holds on the whole set $Y_K$. Therefore $E=E_K$ by the previous theorem, and so $\mathcal E=\mathcal E_K$.
\ep

Let $E$ be the $K$-algebra of locally constant $K^{ab}$-valued $Gal(K^{ab}/K)$-equivariant functions on $Y_K$. Let us now show directly that $E\rtimes I_K$ is an arithmetic subalgebra of $A_K$.\\
In order to prove the density of the $\C$-algebra generated by $E\rtimes I_K$ in $A_K$, it suffices to show that the $\C$-algebra generated by $E$ is equal to the algebra of complex valued locally constant functions on $Y_K$. Since~$Y_K$ is a projective limit of finite $Gal(K^{ab}/K)$-sets, this follows from the following simple statement: if~$L$ is a finite Galois extension of $K$ and $Y$ is a finite $Gal(L/K)$-set, then the $L$-linear span of the $K$-algebra of $Gal(L/K)$-equivariant functions $Y\to L$ coincides with the $L$-algebra of all $L$-valued functions on $Y$.\\
In particular, $E$ separates points of $Y_K$. The property that $K^{ab}$ is generated by the values $f(y)$, $f\in E$, for any $y\in Y^\times_K$, follows now from Lemma~\ref{lgen}, as $Gal(K^{ab}/K)$ acts freely on $Y_K^\times$.\\
Thus $E\rtimes I_K\subset A_K$ is indeed an arithmetic subalgebra. Furthermore, it is easy to see that~$E$ is an inductive limit of \'{e}tale $K$-algebras and $\mathrm{Hom}_{K\text{-alg}}(E,\overline K)=Y_K$. Therefore $\mathcal E=E\rtimes I_K$ is, in fact, an endomotive and $\mathcal E^{an}=A_K$.\\ \\
We finish by making a few remarks about general arithmetic subalgebras of the BC-system $\mathcal A_K$. Assume $\mathcal E\subset A_K$ is an arithmetic subalgebra. Also assume that it contains the elements $U_\aaa$ and $U_\aaa^*$ for all $\aaa\in I_K$. Consider the image of $\mathcal E$ under the canonical conditional expectation $A_K\to C(Y_K)$, and let $E$ be the $K$-algebra generated by this image. Then $E$ satisfies the following properties:
\begin{itemize}
\item[(a$'$)] every function in $E$ is continuous;
\item[(b$'$)] the $\C$-algebra generated by $E$ is dense in $C(Y_K)$; in particular, $E$ separates points of $Y_K$;
\item[(c$'$)] $E$ is invariant under the endomorphisms $\rho_\aaa$ and $\sigma_\aaa$ for all $\aaa\in I_K$;
\item[(d$'$)] for every $f\in E$ we have $f(Y_K^\times)\subset K^{ab}$ and the map $f\colon Y_K^\times\to K^{ab}$ is $Gal(K^{ab}/K)$-equivariant.
\end{itemize}
Conversely, if $E$ is a unital $K$-algebra of functions on $Y_K$ with properties (a$'$)-(d$'$), then $\mathcal E=E\rtimes I_K$ is an arithmetic subalgebra of $A_K$ and the intersection $\mathcal E\cap C(Y_K)$, as well as the image of $\mathcal E$ under the conditional expectation onto $C(Y_K)$, coincides with $E$. Note again that the property that $K^{ab}$ is generated by the values~$f(y)$, $f\in E$, for any $y\in Y^\times_K$, follows from the proof of Lemma~\ref{lgen}. The largest algebra satisfying properties (a$'$)-(d$'$) is the $K$-algebra of continuous functions such that their restrictions to $\aaa Y^\times_K$ are $K^{ab}$-valued and $Gal(K^{ab}/K)$-equivariant for all $\aaa\in I_K$. This algebra is strictly larger than the algebra $E_K$. Indeed, it, for example, contains the functions of the form $\sum^\infty_{n=0}q_n\rho^n_{\pp_v}(1)$, where $\sum_nq_n$ is any convergent series of rational numbers. Such a function takes value $\sum_{n=0}^\infty q_n$, which can be any real number, at every point $y\in \cap_{n\ge0}\pp_v^nY_K$.


\part{Functoriality}
\noindent
In \cite{LNT} Laca, Neshveyev and Trifkovic were able to construct a functor from the category of number fields to the category of BC-systems. In the latter morphisms are given by correspondences in form of a Hilbert $C^*$-bimodule. More precisely, for an inclusion $\sigma: K \to L$ of number fields they construct, quite naturally, an $A_L$-$A_K$ correspondence $Z = Z_{K,\sigma}^L$, i.e. a right Hilbert $A_K$-module $Z$ with a left action of $A_L$ (cf., \eqref{bctype}). Unfortunately, the time evolutions of $A_K$ and $A_L$ are not compatible under $Z$, which is in fact not surprising. In order to remedy the situation the authors of \cite{LNT} introduce a normalized time evolution $\widetilde \sigma _t$ on the $A_K$ given by \eq{\label{normalizedtime}\widetilde \sigma _t (f u_s) = N_{K/\q}(s)^{it / [K:\q]} f u_s \,.}
With this normalization they obtain a functor $K \mapsto (A_K,\widetilde \sigma _t)$ where the correspondences are compatible with the time evolutions. \\ \\
We will show that their functor arises naturally in the context of (algebraic) endomotives. However it doesn't seem likely that the normalized time evolution \eqref{normalizedtime} can be recovered naturally in the framework of endomotives, at least not in a naive sense (see section \ref{timeevolution}). \\ \\
The first obstacle in constructing an algebraic version of the functor constructed in \cite{LNT} is that the different algebraic endomotives $\mathcal E _K$ are defined over different number fields which means that they live in different categories. \\ 
To overcome this, we introduce the notion of "base-change" in this context. More precisely, one finds two natural ways of changing the base of $\mathcal E _K$ which correspond to the two fundamental functoriality properties of class field theory given by the Verlagerung and restriction map respectively. Although both procedures change the algebraic endomotive, the analytic endomotive of the initial and base changed endomotive will remain the same. \\ \\
Our strategy is then, first, to base change all the $\mathcal E _K$ down to $\q$ and then, second, construct a functor from the category of number fields to the category of algebraic endomotives over $\q$. Finally, we will show that our functor recovers the functor constructed in \cite{LNT} (except for the normalization \eqref{normalizedtime}).

\begin{remark}
We would also like to mention the very interesting recent work of Cornelissen and Marcolli \cite{CorMar}, where it is shown that two BC-systems are isomorphic as ("daggered") $C^*$-dynamical systems if and only if the underlying number fields are isomorphic.
\end{remark}

\subsection*{Notations and Conventions}

In the following, when speaking about extensions of number fields, instead of specifying an embedding $\sigma : K \to L$, we simply write $L/K$. 
Moreover we fix a tower $M/L/K$ of finite extensions of number fields (contained in $\complex$). 
We denote the Artin reciprocity map by $[\cdot]_K : \mathbb A _K ^\times \to Gal(K^{ab}/K)$.

\section{Algebraic preliminaries}
\noindent
Recall the two fundamental functoriality properties of Artin's reciprocity map in form of the following two commutative diagrams (cf., \cite{nk})
\diagramnr{diagrams}{\mathbb A _L^\times \ar[rr]^{[\cdot]_L} && \lab & & \mathbb A_L^\times \ar[d]_{N_{L/K}} \ar[rr]^{[\cdot]_L} && \lab \ar[d]^{Res} \\ \mathbb A _K^\times \ar[rr]^{[\cdot]_K} \ar[u]^{i_{K/L}}&& \kab \ar[u]_{Ver} && \mathbb A _K^\times \ar[rr]^{[\cdot]_K} && \kab}
\begin{remark}
Notice that the Verlagerung $Ver$ is injective.
\end{remark} 
\noindent
The diagrams allow one to define two maps of topological monoids (which are of central importance for everything that eventually follows) \eqst{\mathcal V_{L/K} : \okhat \times_{\okhat^\times} \kab \longrightarrow \olhat \times_{\olhat^\times} \lab \ ; [\rho,\alpha] \mapsto [i_{K/L}(\rho),Ver(\alpha)]}
and \eqst{\mathcal N _{L/K} : \olhat \times_{\olhat^\times} \lab  \longrightarrow \okhat \times_{\okhat^\times} \kab \ ; [\gamma,\beta] \mapsto [N_{L/K}(\gamma),Res(\beta)]\,.}

\begin{remark}
The first map is always injective\footnote{This follows from a Galois descent argument.}, whereas the second map is in general neither injective nor surjective.
\end{remark}
\noindent
Now, using these two maps, we can define two "base-change" functors relating the categories $\mathfrak S _{\mathrm{DR}_K}$ and $\mathfrak S _{DR_L}$ (cf.,\ section \ref{bds}). \\
The first functor \eqst{\mathfrak V = \mathfrak V _{L/K} : \mathfrak S_{DR_L} \longrightarrow \mathfrak S_{\mathrm{DR}_K}} is given by sending a finite set $S$ with action of $DR_L$ to the set $S$ with an action of $\mathrm{DR}_K$ given by restricting the action of $DR_L$ via $\mathcal V_{L/K}$.\\
The second functor \eqst{\mathfrak N = \mathfrak N_{L/K} : \mathfrak S _{\mathrm{DR}_K} \longrightarrow \mathfrak S _{DR_L}} is defined by sending a finite set $S$ with its action by $\mathrm{DR}_K$ to the same set $S$ with an action of $DR_L$ defined by pulling back the action of $\mathrm{DR}_K$ via $\mathcal N _{L/K}$. Using the functorial equivalence $\mathfrak S _{\mathrm{DR}_K} \to \mathfrak E _{\Lambda_K}$ (see \eqref{functorBdS}), we obtain corresponding functors on the algebraic side \eqst{\mathfrak V ^{alg} : \mathfrak E _{\Lambda_ L} \longrightarrow \mathfrak E _{\Lambda_K}} and \eqst{\mathfrak N ^{alg} : \mathfrak E _{\Lambda_ K} \longrightarrow \mathfrak E _{\Lambda_L}\,.}

\begin{lemma}
1) The functor $\mathfrak V ^{alg}$ is determined by the fact that a finite abelian extension $L'$ of $L$ is sent to the direct product $\prod_{i=1}^h K'$, where the finite, abelian extension $K'$ of $K$ and the index $h$ are specified below. \\
2) The functor $\mathfrak N ^{alg}$ is given by $E \mapsto E \otimes_K L$.
\end{lemma}
\begin{proof}
1) Define the map $\phi$ to be the composition of the Verlagerung $Gal(K^{ab}/K)\to Gal(L^{ab}/L)$ and the projection $Gal(L^{ab}/L)\to Gal(\widetilde{L'}/L)$ where $\widetilde{L'}$ denotes the Galois closure of $L'$. We can identify the quotient $Gal(K^{ab}/K) / \mathrm{Ker} \phi$ with a finite, abelian Galois group $Gal(\widetilde K / K)$ sitting inside $Gal(\widetilde{L'}/L)$, i.e. $Gal(\widetilde K / K) \cong Gal(\widetilde{L'}/L^K)$ for a subfield $L^K \subset \widetilde{L'}$. We define $K'$ to be the subfield of $\widetilde K$ corresponding to the subgroup $Gal(\widetilde{L'}/L')\cap Gal(\widetilde{L'}/L^K) \subset Gal(\widetilde K / K)$. Using again only basic Galois theory we see that the fraction \eqst{\frac{|Gal(\widetilde{L'}/L)|\cdot |Gal(\widetilde{L'}/L')\cap Gal(\widetilde{L'}/L^K)|}{|Gal(\widetilde{L'}/L^K)|\cdot |Gal(\widetilde{L'}/L') |}} is actually a natural number and this will be the index $h$. In particular we see that we have the equality $|\mathrm{Hom}_L(L',\overline L)| = h \cdot | \mathrm{Hom}_K(K',\overline K)| =  |\mathrm{Hom}_K(\prod_{i=1}^h K',\overline K)|$. \\
2) This is obvious.
\end{proof}

\begin{remark}
If $L' / L$ is Galois then $K' / K$ is also Galois.
\end{remark}
\noindent
Let us make the functor $\mathfrak V ^{alg}$ more transparent in the context of strict ray class fields which occur in the definition of the $\mathcal E _K$. For this, let us first introduce the following notation. If $\mathfrak d$ denotes a non-zero, integral ideal in $\iok$, we denote by $\mathfrak d ^L$ the corresponding ideal in $\iol$. For example, if $\mathfrak d = \mathfrak p$ is a prime ideal then $\mathfrak p ^L = \mathfrak p \mathcal O _L$ is usually written in the form \eqst{\mathfrak p^L = \prod_{\mathfrak P | \mathfrak p} \mathfrak P ^{e(\mathfrak P | \mathfrak p)}\, ,}
where $\mathfrak P$ denotes a prime ideal of $L$ and $e(\mathfrak P | \mathfrak p)$ the ramification index of $\mathfrak P$ in $\mathfrak p$. Moreover, let us denote by $K_\mathfrak d$ and $L_{\mathfrak d ^L}$ the corresponding strict ray class fields and by $K^{\mathfrak d ^L}$ the field constructed from $L_{\mathfrak d ^L}$ above. Then we have the following: 
\begin{lemma}
With the notations from above let $\mathfrak d$ be in $\iok$. Then we have \eqst{K_\mathfrak d = K^{\mathfrak d ^L} \subset L_{\mathfrak d ^L}\,.}
\end{lemma} 
\begin{proof}
Using basic class field theory (cf., \cite{nk}) the two assertions can be reformulated in the idelic language and are seen to be equivalent to \eqst{ \iota_{L/K}^{-1}(C_L^{\mathfrak d ^L}) = C_K^\mathfrak d \text{ \ \ \ \ and \ \ \  \   }\iota_{L/K}(N_{L/K}(C_L^{\mathfrak d ^L}) ) \subset C_L^{\mathfrak d ^L}\, ,}
where $C_K^{\mathfrak d}$ is the standard open subset of $C_K = \mathbb A _K^\times / K^\times$ such that $C_K / C_K^\mathfrak d \cong Gal(K_\mathfrak d / K)$ and analogously for $C_L^{\mathfrak d ^L}$. Further, it is enough to consider the case $\mathfrak d = \mathfrak p ^i$ for some $i\geq 1$. Let us recall the following fact from ramification theory. If $\mathfrak P$ divides $\mathfrak p$ with ramification index $e = e(\mathfrak P | \mathfrak p)$ and if we denote by $\iota_\mathfrak P : K_\mathfrak p \to L _\mathfrak P$ the natural inclusion of local fields, we have \eqst{\iota_\mathfrak P ^{-1}(\mathfrak P^{ei} \mathcal O _{L _\mathfrak P}) = \mathfrak p ^i \mathcal O _{K_\mathfrak p} \,.}
This proves the first assertion, and the second assertion follows directly from the definition of the norm map $N_{L/K}$.
\end{proof}
\noindent
As a first application, it is shown in the next proposition how one can relate the different algebras $E_K$ (cf., \ref{defek}).

\begin{proposition}
\label{fruitprop}
1) The functor $\mathfrak V ^{alg}$ induces a $K$-algebra homomorphism $\mathfrak V^{alg}(E_L) \to E_K$ compatible with the $\Lambda_K$-structure.\\
2) The functor $\mathfrak N ^{alg}$ induces a $L$-algebra homomorphism $E_K \otimes_K L \to E_L$ compatible with the $\Lambda_L$-structure. \\
3) There exists an injective $K$-algebra homomorphism $E_K \to E_L$.
\end{proposition}
\begin{proof}
Using the two commutative diagrams
\diagram{DR_L \ar@(dl,dr)[]_{\mathcal V _{L/K}} & \mathrm{DR}_K \ar@(dl,dr)[] \ar[l]_{\mathcal V _{L/K}} & & \mathrm{DR}_K \ar@(dl,dr)[]_{\mathcal N _{L/K}} &  DR_L \ar[l]_{\mathcal N _{L/K}} \ar@(dl,dr)[] \\ \mathrm{DR}_K & \mathrm{DR}_K \ar@{=}[l] & & DR_L &  DR_L \ar@{=}[l]}
the \textbf{first two} assertions follow immediately if we can show that $\mathcal V _{L/K}: \mathrm{DR}_K \to DR_L$ and $\mathcal N _{L/K}: DR_L \to \mathrm{DR}_K$ are compatible with the profinite structures of $\mathrm{DR}_K$ and $DR_L$.
In this case we can simply apply the equivalence \eqref{profunctor}.
The compatibility of $\mathcal  N _{L/K}$ with the profinite structure of $DR_L$ follows from the compatibility of $\mathcal V _{L/K}$ with the profinite structure of $\mathrm{DR}_K$, and this would follow if $\mathcal V _{L/K}$ factors over \eqst{\ok / \mathfrak f \times_{(\ok / \mathfrak f)^\times} C_\mathfrak f \to \ol / \mathfrak f ^L \times_{(\ol / \mathfrak f ^L)^\times} C_{\mathfrak f ^L}\,.}
But this is true thanks to our previous lemma. \\
To prove the \textbf{third} assertion we define a surjective map \eqst{\omega_\mathfrak f : \{\mathfrak D \mid \mathfrak D \text{ divides } \mathfrak f ^L \} \to \{ \mathfrak d \mid \mathfrak d \text{ divides } \mathfrak f \}} by \eqst{\mathfrak D \longmapsto \prod_{\mathfrak p \mid \mathfrak f}\mathfrak p ^{max\{ j \  : \ \mathfrak P^{j e(\mathfrak P \mid \mathfrak p) } \mid \mathfrak D \ \forall \ \mathfrak P \mid \mathfrak p\}}\,.} Now we can define an embedding of $K$-algebras \eqst{E_{K,\mathfrak f}=\prod_{\mathfrak d | \mathfrak f}K_\mathfrak d \longrightarrow E_{L,\mathfrak f ^L} = \prod_{\mathfrak D | \mathfrak f ^L} L_{\mathfrak D}} by embedding $K_\mathfrak d$ into $L_\mathfrak D$ whenever $\omega_\mathfrak f(\mathfrak D) = \mathfrak d$. It is not very difficult to check that these maps induce a $K$-algebra embedding of the corresponding inductive systems. 
\end{proof}
\begin{remark}
Due to the fact that $\mathfrak f ^M = (\mathfrak f ^L)^M$ we see that the third map of the last proposition is in fact functorial, i.e., the composition $E_K \to E_L \to E_M$ equals $E_K \to E_M$. But, on the other hand, the inclusion $E_K \to E_L$ is not compatible with any $\Lambda$-structure.
\end{remark}

\section{On correspondences of endomotives}
\noindent
In this section we will follow our main reference \cite{ConMar} pp. 594.

\subsection{Algebraic correspondences}
\noindent
An algebraic endomotive $\mathcal E  = A \rtimes S$ can be described equivalently as a groupoid $\mathcal G$ as follows. Let us introduce, for $s = \rho_1 / \rho_2 \in \widetilde S$, two projections $E(s) = \rho_1^{-1}(\rho_2(1)\rho_1(1))$ and $F(s) = \rho_2^{-1}(\rho_2(1)\rho_1(1))$. They satisfy the relations $E(s^{-1})=F(s)=s(E(s))$ and show up naturally in that they are the biggest projections such that $s : A_{E(s)} = E(s)A \to A_{F(s)}$ is an isomorphism. Now, as a set $\mathcal G$ is defined by \eqst{\mathcal G = \mathrm{Spec}(\bigoplus_{s \in \widetilde S} A_{F(s)}) = \sqcup_{s\in \widetilde S} \mathrm{Spec}(A_{F(s)})\,.} The range and source maps \eqst{r,s : \mathrm{Spec}(\bigoplus_{s\in \widetilde S} A_{F(s)}) \to \mathrm{Spec}(A)} are given by the natural projection $A \to A_{F(s)}$, and the natural projection composed with the antipode $P : \bigoplus A_{F(s)} \to \bigoplus A_{F(s)}$ given by \eqst{P(a)_s = s(a_{s^{-1}}), \ \ \forall s \in \widetilde S \,.}
An \textbf{algebraic correspondence} between two algebraic endomotives $\mathcal E ' = A'\rtimes S'$ and $\mathcal E = A \rtimes S$ is given by a disjoint union of zero-dimensional pro-varietes $\mathcal Z = \mathrm{Spec}(C)$ together with compatible left and right actions of $\mathcal G '$ and $\mathcal G$ respectively. A right action of $\mathcal G$ on $\mathcal Z$ is given by a continuous map \eqst{g : \mathrm{Spec}(C) \to \mathrm{Spec}(A)} together with a family of partial isomorphisms
\eqst{z \in g^{-1}(\mathrm{Spec}(A_{E(s)})) \mapsto z \cdot s \in g^{-1}(\mathrm{Spec}(A_{F(s)})) \ \ \ \ \forall s \in \widetilde S}
satisfying the obvious rules for a partial action of an abelian group (cf., \cite{ConMar} p. 597). Analogously, one defines a left action. It is straightforward to check that a left (resp. right) action of $\mathcal G$ on $\mathcal Z$ is equivalent to a left (resp. right) $\mathcal E$-module structure on $C$. \\
The composition of algebraic correspondences is given by the fibre product over a groupoid. On the algebraic side this corresponds to the tensor product over a ring.
\begin{remark}
The main advantage of using the groupoid language comes from the fact that it provides a natural framework for constructing so called analytic correspondences $\mathcal Z ^{an}$ between $\mathcal E ' {}^{an}$ and $\mathcal E ^{an}$ out of algebraic correspondences. In fact, the procedure is functorial (see Thm. 4.34 \cite{ConMar})
\end{remark} 
\noindent
In our reference \cite{ConMar}, morphisms of the category of algebraic endomotives over $K$ are defined in terms of \'etale correspondences, where $\mathcal Z$ is \'etale if it is finite, and projective as a right module. We shall eventually see that the finiteness condition is too restrictive for our applications. Nevertheless, the functorial assignment $\mathcal Z \mapsto \mathcal Z ^{an}$ has a domain much larger than only \'etale correspondences, containing in particular the algebraic correspondences occurring in our applications. In summary, we enlarge tacitly the morphisms in the category of algebraic endomotives by allowing those contained in the domain of the assignment $\mathcal Z \mapsto \mathcal Z ^{an}$. 

\subsection{Analytic correspondences}
\noindent
As we have already seen in the section \ref{anflavour} about analytic endomotives, the (functorial) transition between algebraic and analytic endomotives is based on the functor $X \mapsto X(\overline K)$ taking $\overline K$-valued points. \\
Given an algebraic endomotive $\mathcal E$ with corresponding groupoid $\mathcal G$, we define the analytic endomotive $\mathcal G ^{an}$ to be the totally disconnected locally compact space $\mathcal G(\overline K)$ of $\overline K$-valued points of $\mathcal G$. An element of $\mathcal G ^{an}$ is therefore given by a pair $(\chi,s)$ with $s \in \widetilde S$ and $\chi$ a character of the (reduced) algebra $A_{F(s)}$, i.e. $\chi(F(s))=1$. The range and source maps \eqst{r,s : \mathcal G ^{an} \to \mathrm{Hom}_{K\text{-}\mathrm{alg}}(A,\overline K)} are given by \eqst{r(\chi,s) = \chi \text{ \ \ \ \ \ \ \ and \ \ \ \ \ \ \ } s(\chi,s) = \chi \circ s\,.}
One shows that $\mathcal E ^{an} = C(\mathrm{Hom}_{K\text{-}\mathrm{alg}}(A,\overline K))$ is isomorphic to the groupoid $C^*$-algebra $C^*(\mathcal G ^{an})$. \\ \\
Now, given an algebraic correspondence $\mathcal Z$ between $\mathcal E '$ and $\mathcal E$, i.e., we have (for the right action) a continuous map \eqst{g : \mathcal Z \to \mathrm{Spec}(A)\, ,} together with partial isomorphisms, we obtain, by taking the $\overline K$-valued points, a continuous map of totally disconnected locally compact spaces \eqst{g_K = g(\overline K) : \mathcal Z (\overline K) = \mathrm{Hom}_{K\text{-}\mathrm{alg}}(C,\overline K) \to \mathrm{Hom}_{K\text{-}\mathrm{alg}}(A,\overline K)\, ,} together with partial isomorphisms \eqst{z \in g_K^{-1}(\mathrm{Hom}_{K\text{-}\mathrm{alg}}(A_{F(s)} ,\overline K) ) \mapsto z \circ s \in g_K^{-1}(\mathrm{Hom}_{K\text{-}\mathrm{alg}}(A_{E(s)},\overline K ) ) \, ,} fulfilling again the obvious rules. \\
As in the algebraic case, this right action of $\mathcal G ^{an}$ on $\mathcal Z (\overline K)$ gives the space of continuous and compactly supported functions $C_c(\mathcal Z (\overline K))$ on $\mathcal Z (\overline K)$ the structure of a right $C_c(\mathcal G ^{an})$-module. Moreover, \textbf{if} the fibers of $g_K$ are discrete (and countable) there is a natural way of defining a $C_c(\mathcal G ^{an})$-valued inner product on $C_c(\mathcal Z (\overline K))$ by setting \eqst{\langle\xi,\eta\rangle(\chi,s) = \sum_{z \in g_K^{-1}(\chi)} \overline \xi (z) \eta(z \circ s)\,.}
In this case we obtain a right Hilbert-$C^*$-module $\mathcal Z ^{an}$ over $C^*(\mathcal G ^{an})$ by completion. Together with the left action $\mathcal Z ^{an}$ becomes a $C^*(\mathcal G ' {}^{an})$-$C^*(\mathcal G ^{an})$ Hilbert-$C^*$-bimodule.

\subsection{Examples}
\label{examples}
\noindent
1) Every algebraic endomotive $\mathcal E$ is a correspondence over itself. In particular the inner product is given on $\mathcal E ^{an} = C(\mathcal X) \rtimes S$ simply by \eqst{\langle\xi,\eta\rangle = \xi ^* \eta \ \ \ \ \forall \xi,\eta \in \mathcal E ^{an} \,.}
2) Let $S \subset T$ be an inclusion of abelian semigroups. Then the algebraic endomotive $K[T] = K\rtimes T$ is naturally a $K[T]$-$K[S]$ correspondence with the obvious continuous map $g: \mathrm{Spec}(\bigoplus_{t \in \widetilde T}K) \to \mathrm{Spec}(K)$ and partial isomorphisms. If we denote the corresponding analytic endomotives by $C^*(T)$ and $C^*(S)$, which are related by the natural conditional expectation $E : C^*(T)\to C^*(S)$ induced by $t \in T \mapsto \left\{ \begin{matrix} t & \text{ if $t\in S$} \\ 0 & \text{ otherwise} \end{matrix}\right.$, we see that the $C^*(S)$-valued inner product on $C^*(T)$ is given by \eqst{\langle\xi,\eta\rangle = E(\xi^*\eta),\ \ \ \ \forall \xi,\eta \in \mathcal C^*(T)\,.}

\section{On base-change}
\label{basechange}
\noindent
Let us start with the data defining our algebraic endomotive $\mathcal E _L$, namely the inductive system $(E_\mathfrak f)_{\mathfrak f \in \iol}$ and the collection of "Frobenius lifts" $\sigma_\mathfrak d$ (cf., \eqref{rhodef}), where the latter define of course the $\rho_\mathfrak d$ but are better suited for the functors $\mathfrak V _{L/K}^{alg}$ and $\mathfrak N _{M/L}^{alg}$ due to their level preserving property. Let us concentrate on the functor $\mathfrak V = \mathfrak V_{L/K}^{alg}$, the arguments for $\mathfrak N _{M/L}^{alg}$ are analogous. Define the $K$-algebras $\widetilde E _\mathfrak f = \mathfrak V (E_\mathfrak f)$, $\widetilde E _L = \varinjlim \widetilde E _\mathfrak f$ and the $K$-algebra homomorphisms $\widetilde \sigma _\mathfrak d = \mathfrak V(\sigma_\mathfrak d) : \widetilde E _L \to \widetilde E _L$. Due to the fact that (cf., \eqref{profunctor}) \eqst{\mathfrak H _L(\widetilde E _L) = DR_L} and \eqst{\mathfrak H _L(\widetilde \sigma _\mathfrak d) = \sigma_\mathfrak d : DR_L \to DR_L\, ,}
the same arguments as in section \ref{E_K} show the existence of projections $\widetilde \pi _\mathfrak d$ and endomorphisms $\widetilde \rho _\mathfrak d$ of $\widetilde E _L$ such that \eqst{\mathcal E _L ^K = \mathfrak V_{L/K}^{alg}(\mathcal E _L)=((\widetilde E _\mathfrak f),\widetilde \iol)} is in fact an algebraic endomotive over $K$. Analogously we construct \eqst{\mathcal E _L^M = \mathfrak N _{M/L}^{alg}(\mathcal E _L)} and obtain in summary the following base-change properties of our algebraic endomotives $\mathcal E _L$. 

\begin{proposition} With the notations from above we have that $\mathcal E _L ^K$ and $\mathcal E _L ^M$ are algebraic endomotives over $K$ and $M$, respectively. Moreover, on the analytic level we have \eqst{(\mathcal E _L^K)^{an} = \mathcal E _L ^{an} = (\mathcal E _L ^M)^{an}\,.}
\end{proposition} 

\begin{remark}
Both assignments are functorial. 
\end{remark}

\section{A functor, a pseudo functor and proof of Theorem \ref{ThmC}}

\subsection{Going down to $\q$}
\label{goingdowntoq}
The base-change mechanism from the last section enables us now to construct a functor from the category of number fields to the category of algebraic endomotives over $\q$ which sends a number field $K$ to $\mathcal E _K ^\q$. Unfortunately, it is \underline{not} possible to construct an algebra homomorphism between $\mathcal E _K^\q$ and $\mathcal E_L^\q$ because the actions of $\iok$ and $\iol$ are not compatible. Instead, given an extension $L/K$ we construct an algebraic $\mathcal E _L^\q$-$\mathcal E _K^\q$ correspondence $\mathcal Z_K^L$ as follows. Recall the examples \ref{examples}. From the first one, we see that we can regard $\mathcal E _K^\q$ as a $\q[\iok]$-$\mathcal E _k ^\q$ correspondence, because we have naturally the inclusion $\q[\iok] \subset \mathcal E _K^\q$. Using the second example in the case of the inclusion $I_K \subset I_L$ we obtain the $\q[\iol]$-$\q[\iok]$ correspondence $\q[\iol]$. In performing the fibre product over $\q[\iok]$ we obtain the $\q[\iol]$-$\mathcal E _K^\q$ correspondence $\mathcal Z _K^L = \q[\iol] \times_{\q[\iok]} \mathcal E _K^\q$ which can be described algebraically by \eqst{\label{algmorph}\mathcal Z _K ^L = \q[\iol] \otimes_{\q[\iok]} \mathcal E _K^\q\,.}
We want to show that there is a natural left action of $\mathcal E _L^\q$ making $\mathcal Z_K^L$ the desired $\mathcal E _L^\q$-$\mathcal E _K^\q$ correspondence. Namely, using the same arguments as in proposition \ref{fruitprop}, we obtain a $\q$-algebra homomorphism $\phi : \mathfrak V _{L/\q}^{alg}(E_L) \to \mathfrak V _{K/\q}(E_K)$ which is furthermore compatible with the $I_K$-actions on both algebras induced by functoriality from the actions of $\mathrm{DR}_K$ and $\mathcal V _{L/K}(\mathrm{DR}_K)$ on $\mathrm{DR}_K$ and $DR_L$, respectively. Thus, we see that \eqst{e U_s \cdot (U_t \otimes f) = U_{st} \otimes \phi(\widetilde e) f\, ,} for $s,t \in \iol$, $e \in \mathfrak V _{L/\q}^{alg}(E_L)$, $f \in \mathcal E_K^\q$ and $\widetilde e$ defined by the equation $eU_{st} = U_{st}\widetilde e \in \mathcal E _L^\q$, gives a well-defined left $\mathcal E _L^\q$-module structure on $\mathcal Z _K ^L$. \\ \\
We can now prove the main result of this chapter.
\begin{theorem}
1) The assignments $K \mapsto \mathcal E _K^\q$ and $L/K \mapsto \mathcal Z _K^L$ define a (contravariant) functor from the category of number fields to the category of algebraic endomotives over $\q$. \\
2) The corresponding functor given by $K \mapsto (\mathcal E _K^L)^{an}$ and $L/K \mapsto (\mathcal Z _K^L)^{an}$ from the category of number fields to the category of analytic endomotives is equivalent to the functor constructed by Laca, Neshveyev and Trifkovic in Thm. 4.4 \cite{LNT}. 
\end{theorem}
\begin{proof}
1) One only has to show that $\mathcal Z _L ^M \otimes _{\mathcal E _L^\q} \mathcal Z _K^L \cong \mathcal Z _K^M$, which is obvious. \\
2) One can check without difficulties that $(\mathcal Z _K^L)^{an}$ is given as a Hilbert $C^*$-module by the inner tensor product of the right $C^*(\iok)$-module $C^*(\iol)$ and the right $\mathcal E_K ^{an}$-module $\mathcal E _K ^{an}$ with its natural left action of $C^*(\iok)$, i.e., \eqst{(\mathcal Z_K^L)^{an} = C^*(\iol) \otimes _{C^*(\iok)} \mathcal E _K^{an}\, ,} and this is exactly the same correspondence as constructed in Theorem 4.4 of \cite{LNT}. 
\end{proof}

\begin{remark}
We see that $\mathcal Z _K ^L$ is not an \'etale correspondence because the complement of $\iok$ in $\iol$ is infinite. Nevertheless, the definition of $\mathcal Z_K^L$ seems to be the most natural one under the circumstances that it is not possible to define interesting algebra homomorphisms between $\mathcal E _K ^{(\q)}$ and $\mathcal E_L ^{(\q)}$, which comes from the fact that Verlagerung and Restriction are not inverse to each other in general, and therefore the actions of $\iok$ and $\iol$ are not compatible. 
\end{remark}

\subsection{$\overline \q$ is too big}
\noindent
In analogy with the last section, where we constructed algebraic correspondences using the base-change induced by the functor $\mathfrak V _{L/K}^{alg}$, one can also use the functor $\mathfrak N _{L/K}^{alg}$ to construct bimodules of algebraic endomotives. \\ \\
Again, by proposition \ref{fruitprop}, we see that \eqst{\mathcal Y_K^L = L[I_K]\otimes_{L[I_L]} \mathcal E _L } is an $\mathcal E_K^L$-$\mathcal E_L$ correspondence. The right $L[I_L]$-module structure of $L[I_K]$ is induced by the norm map $\iol \to \iok$. \\
But this time, we do not obtain a functor. Of course, one can check that for a tower $M/L/K$ of number fields we have an isomorphism of $\mathcal E _K^M$-$\mathcal E _M$ bimodules \eqst{(M[I_K]\otimes_{M[I_L]}\mathcal E_L^M) \otimes_{\mathcal E _L^M} (M[I_L]\otimes_{M[I_M]}\mathcal E _M) \cong M[I_K]\otimes_{M[I_M]} \mathcal E _M\, ,} but in order to make this functorial for all number fields, we would have to make sense of a $\Lambda$-structure over $\overline \q$ which is compatible with $\Lambda$-structures over number fields and this does not seem likely to the author.

\subsection{On the time evolution}
\label{timeevolution}
\noindent
In this section we would like to make some remarks about the question of whether the normalized time evolution \eqref{normalizedtime} introduced in \cite{LNT} fits into the framework of endomotives. \\ \\
Due to the fact that the analytic endomotive of the base-changed algebraic endomotive $\mathcal E _K^\q$ is equal to $\mathcal E _K^{an}$ we see in particular that $\mathcal E _K ^\q$ is an uniform endomotive (over $\q$) with the same measure $\mu_K$ as the natural measure of $\mathcal E _K$.  So, in particular the base-changed endomotive $\mathcal E_ K^\q$ does not recover the normalized time evolution, if one tries to define the time evolution on $\mathcal E _K ^\q$ by means of normalized counting measures. This is clear, because the normalized norm $\widetilde{N} =N_{K /\q}^{1/[K:\q]}$ used in \cite{LNT} is no longer rational-valued on ideals of $K$, so $\widetilde{N}$ cannot arise from a counting procedure as one can for the usual norm $N _{K/ \q}$.  This shows that in order to extend the base-change $\mathcal E _K \mapsto \mathcal E _K ^\q$ in a way such that the normalized time evolution appears on $(\mathcal E _K^\q)^{an}$ one has to find a natural method of assigning to $\mu_K$ a measure $\mu_K^\q$ which recovers the normalized time evolution\footnote{The methods of \cite{LLN} show that such a measure should exist and is in fact determined by the normalized norm $\widetilde{N}$.}. We have argued that this cannot be done in the naive sense, but it would surely be interesting to find a natural method solving this problem.

\appendix

\section{Compatibility of symmetries with other constructions}
\noindent
We would like to clarify the relation between the different definitions of symmetries of Bost-Connes systems occurring in the literature.  \\
In \cite{LLN} or in the framework of endomotives, as in our work, symmetries are always given by automorphisms, on the other hand, e.g., in \cite{cmr} symmetries occur also in form of endomorphisms. \\
Apart from the two natural actions used to define the Bost-Connes system $\mathcal A _K$ in form of the action of $\iok = \okhat ^\natural / \okhat^\times$ on $Y_K = \okhat \times_{\okhat^\times}\kab$ by \eqst{s \cdot [\rho,\alpha] = [\rho s, [s]^{-1}\alpha]} and the action of $\kab$ on $Y_K$ given by \eqst{\gamma \cdot [\rho,\alpha] = [\rho,\gamma \alpha]\, ,} there is a third natural action of $\okhat ^\natural$ on $Y_K$ given by \eqst{s \star [\rho,\alpha] = [\rho s , \alpha]\,.}
In this way we get an action of $\kab$ as automorphisms on $C(Y_K)$ by \eqst{{}^{\gamma}\hspace{-0.6mm}f([\rho,\alpha]) = f([\rho,\gamma^{-1}\alpha])} and an action of $\okhat ^\natural$ on $C(Y_K)$ as endomorphisms by \eqst{{}^{s\star}\hspace{-0.6mm}f([\rho,\alpha]) = \left\{ \begin{array}{cc} f([\rho s^{-1},\alpha]), & \text{if } \rho s^{-1} \in \okhat \\ 0 & \text{otherwise} \end{array} \right. .}
The latter action is used for example in \cite{cmr} to define the symmetries of the corresponding Bost-Connes systems. The two notions of symmetries are related as follows. If we take $s \in \okhat ^\natural$, denote by $\gamma = [s] \in \kab$ its image under Artin's reciprocity map and by $\overline s \in \iok$ the associated integral ideal, we see that for every function $f \in C(Y_K)$ the following relation holds \eq{{}^{s\star}\hspace{-0.6mm}f(\overline s \cdot [\rho,\alpha]) = {}^{\gamma}\hspace{-0.6mm}f([\rho,\alpha])\,.}
This explains why both definitions of symmetries induce the same action on extremal $KMS_\beta$-states, for $\beta > 1$, and on extremal $KMS_\infty$-states evaluated on the arithmetic subalgebra. 
\begin{remark}
One does immediately see that the strict ray class group $Cl_K^+ = \kab / [\okhat^\times]$ of $K$ is responsible for the fact that $\okhat ^\natural$ acts by endomorphisms on $C(Y_K)$. If the strict ray class group of $K$ is trivial, then $\okhat^\natural$ acts by automorphisms as well and the actions of $\kab$ and $\okhat ^\natural$ agree, in fact.
\end{remark}

\section{On Euler's formula}
\label{appeuler} \noindent
In the following we show that the classical Euler totient function can be naturally generalized to arbitrary number fields. This is surely a well-known result.
\begin{lemma}
\label{eulertotient}
For $K$ a number field define the function $\varphi_K : \iok \to \mathbb N$ by setting \eqst{\varphi_K(\mathfrak f) = |(\mathcal O _K / \mathfrak f)^\times|\,.}
Then the following equality holds \eqst{N(\mathfrak f) = |\mathcal O _K / \mathfrak f| = \sum_{\mathfrak d | \mathfrak f}\varphi_K(\mathfrak d)\,.}
\end{lemma}
\begin{proof}
Thanks to the Chinese remainder theorem, it is enough to show $\varphi_K(\mathfrak p ^k) = N(\mathfrak p ^k)-N(\mathfrak p^{k-1})$ for all $k \geq 1$. Using the fact that $\ok / \mathfrak p ^k$ is a local ring with maximal ideal $\mathfrak p / \mathfrak p ^k$ we obtain $\varphi_K(\mathfrak p ^k)=  |\ok / \mathfrak p^k| - |\mathfrak p / \mathfrak p ^k| = N(\mathfrak p^{k}) - |\mathfrak p / \mathfrak p ^k|$. The isomorphism $(\ok/\mathfrak p ^k) / (\mathfrak p / \mathfrak p ^k) \cong \ok / \mathfrak p$ and the multiplicativity of the norm imply $|\mathfrak p / \mathfrak p ^k| = N(\mathfrak p ^{k-1})$ which finishes the proof.
\end{proof}

\bibliographystyle{plain}
\bibliography{endobck}

\end{document}